\documentclass{amsart} 
\usepackage[all,cmtip]{xy}

\newtheorem{mytheorem}{Theorem}[section]{\bf}{\it}
\numberwithin{mytheorem}{section}
\newtheorem{mycorollary}[mytheorem]{Corollary}{\bf}{\it}
\newtheorem{mydefinition}[mytheorem]{Definition}{\bf}{}
\newtheorem{myprop}[mytheorem]{Proposition}{\bf}{\it}
{\bf}{\it}
\newtheorem{myproblem}[mytheorem]{Problem}{\bf}{\it}
\newtheorem{myremark}[mytheorem]{Remark}{\bf}{}
\newtheorem{myexample}[mytheorem]{Example}{\bf}{\it}

\newcommand{\IR}{\mathbb R}
\newcommand{\IN}{\mathbb N}
\newcommand{\II}{\mathbb I}
\newcommand{\K}{\mathcal K}
\newcommand{\U}{\mathcal U}
\newcommand{\V}{\mathcal V}
\newcommand{\F}{\mathcal F}
\newcommand{\X}{\mathcal X}
\newcommand{\C}{\mathcal C}

\newcommand{\M}{\mathfrak M}
\newcommand{\Ra}{\Rightarrow}
\newcommand{\w}{\omega}
\newcommand{\cf}{\mathrm{cf}}

\newcommand{\cc}{\mathbf{C}_k}
\newcommand{\CC}{C_k}

\newcommand{\supp}{\mathrm{supp}}
\newcommand{\cl}{\mathrm{cl}}

\newcommand{\DD}{\mathcal{D}}
\newcommand{\KK}{\mathcal{K}}
\newcommand{\NN}{\mathbb{N}}
\newcommand{\FF}{\mathcal{F}}
\newcommand{\I}{\mathcal{I}}

\newcommand{\e}{\varepsilon}
\newcommand{\Pp}{\mathfrak{P}}
\newcommand{\Nn}{\mathcal{N}}

\title[On the $\CC$-stable closure of the class of (separable) metrizable spaces]{On the $\CC$-stable closure of the class\\ of (separable) metrizable spaces}
\thanks{The second named author was partially supported by Israel Science Foundation grant 1/12. }

\author{Taras Banakh \and Saak Gabriyelyan}

\address{T.Banakh:
            Ivan Franko National University, Lviv, Ukraine, and Jan Kochanowski University, Kielce, Poland}
\email{t.o.banakh@gmail.com}
\address{S. Gabriyelyan:
              Department of Mathematics, Ben-Gurion University of the Negev, Beer Sheva, Israel}
\email{saak@math.bgu.ac.il}           %

\begin{document}

\begin{abstract}
Denote by $\cc[\M]$ the $\CC$-stable closure of the class $\M$ of all metrizable spaces, i.e., $\cc[\M]$ is  the smallest class of topological spaces that contains $\M$ and is closed under taking subspaces, homeomorphic images, countable topological sums, countable Tychonoff products, and function spaces $\CC(X,Y)$ with Lindel\"of domain
in this class. We show that the class  $\cc[\M]$ coincides with the class of all topological spaces homeomorphic to subspaces of the function spaces $\CC(X,Y)$ with a  separable metrizable space $X$ and a metrizable space $Y$.
We say that a topological space $Z$ is Ascoli if every compact subset of $\CC(Z)$ is evenly continuous; by the Ascoli Theorem, each $k$-space is Ascoli.
We prove that the class $\cc[\M]$ properly contains the class of all Ascoli $\aleph_0$-spaces and is  properly  contained in the class of $\Pp$-spaces, recently introduced by Gabriyelyan and K\c akol.
Consequently, an Ascoli space $Z$ embeds into the function space $\CC(X,Y)$ for suitable separable metrizable spaces $X$ and $Y$ if and only if $Z$ is an $\aleph_0$-space.
\end{abstract}
\maketitle

\section{Introduction}

All topological spaces considered in this paper are regular $T_1$-spaces. For two topological spaces $X$ and $Y$, we denote by  $\CC(X,Y)$ the space $C(X,Y)$ of all continuous functions  from $X$ into $Y$ endowed with the compact-open topology. The space $\CC(X,\IR)$ of all real-valued functions on $X$ is denoted by $\CC(X)$. In \cite{Mich} Michael introduced the class of $\aleph_0$-spaces, which appeared to be essential in studying function spaces $C_k(X,Y)$, see \cite{mcoy}. Recall that a topological space $X$ is an {\em $\aleph_0$-space} if $X$ possesses a countable $k$-network. A family $\mathcal{N}$ of subsets of $X$ is a {\em $k$-network} in $X$ if for any open subset $U\subset X$ and compact subset $K\subset U$ there exists a finite subfamily $\F\subset\mathcal{N}$ such that $K\subset\bigcup\FF\subset U$. Any separable metrizable space is an $\aleph_0$-space. In \cite{Mich} Michael proved that for any $\aleph_0$-spaces $X$ and $Y$ the function space $\CC(X,Y)$ is also an $\aleph_0$-space. In particular, for any separable metrizable spaces $X$ and $Y$ the function space $\CC(X,Y)$ is not necessarily metrizable, but it is always an $\aleph_0$-space.

Having in mind the Nagata-Smirnov metrization theorem,  O'Meara \cite{OMe2} introduced the class of $\aleph$-spaces,  which contains all metrizable spaces and all $\aleph_0$-spaces. A topological space $X$ is called an {\em $\aleph$-space} if $X$ is regular and  possesses a $\sigma$-locally finite $k$-network.  It follows that a topological space is an $\aleph_0$-space if and only if it is a Lindel\"of $\aleph$-space. Foged \cite{foged} (and O'Meara \cite{OMe2})  generalized the result of Michael proving that for any Lindel\"of $\aleph$-space $X$ and any (paracompact) $\aleph$-space $Y$ the function space $C_k(X,Y)$ is a (paracompact) $\aleph$-space.
It is well-known (see \cite{Mich,gruenhage}) that the classes of  $\aleph$-spaces and \mbox{$\aleph_0$-spaces} are closed under taking subspaces, homeomorphic images, countable topological sums and countable Tychonoff products. These stability properties motivate the following definition.
\begin{mydefinition}
A class $\X$ of topological spaces is said to be
\begin{itemize}
\item[$\bullet$] {\em stable} if $\X$ is closed under the operations of taking subspaces, homeomorphic images, countable topological sums, and  countable Tychonoff products;
\item[$\bullet$] {\em $\CC$-stable} if $\X$ is stable and for any Lindel\"of space $X\in\X$ and any space $Y\in\X$ the function space $\CC(X,Y)$ belongs to the class $\X$.
\end{itemize}
\end{mydefinition}

Let us observe that each stable class $\X$ containing a non-empty topological space contains all zero-dimensional separable metrizable spaces, and if the closed interval $\II =[0,1]$ belongs to $\X$, then $\X$ contains all separable metrizable spaces (see Proposition \ref{p:Ascoli} below).

For two classes $\X$ and $\tilde\X$ of topological spaces, we shall say that $\tilde\X$ is an {\em extension} of $\X$ if $\X\subset\tilde\X$. Among all ($\CC$-)stable extensions of $\X$ there is the smallest one.
\begin{mydefinition}
For a class $\X$ of topological spaces, denote  by $[\X]$ and $\cc[\X]$ the smallest stable extension and  the smallest $\CC$-stable  extension of $\X$, respectively. The classes $[\X]$ and $\cc[\X]$ will be called the {\em stable closure} and the {\em $\CC$-stable closure} of the class $\X$, respectively.
\end{mydefinition}

In the paper we study the $\CC$-stable closures $\cc[\M]$ and $\cc[\M_0]$ of the classes of all metrizable spaces $\M$ and all separable metrizable spaces $\M_0$, respectively.

For two classes $\X$ and $\mathcal{Y}$ of topological spaces we denote by $\cc(\X,\mathcal{Y})$ the class of all topological spaces  which embed into function spaces $\CC(X,Y)$ with $X\in\X$ and $Y\in\mathcal{Y}$. In Section \ref{sec-Charac-C(M)} we obtain the following characterization of the classes $\cc[\M]$ and $\cc[\M_0]$.

\begin{mytheorem} \label{t-Main-C(M)}
$\cc[\M_0]=\cc(\M_0,\M_0)$ and $\cc[\M]=\cc(\M_0,\M)$.
\end{mytheorem}
In other words, Theorem \ref{t-Main-C(M)} means that the problem of whether a topological space $Z$ belongs to the class $\cc[\M]$ (respectively, $\cc[\M_0]$) is equivalent to the embedding problem of the space $Z$ into the function space $\CC(X,Y)$ with a separable metrizable space $X$ and a (respectively, separable) metrizable space $Y$.

The aforementioned results of Michael imply that the class of $\aleph_0$-spaces is a $\CC$-stable extension of the class $\M_0$. In \cite{Banakh} the first named author introduced another $\CC$-stable extension of the class $\M_0$, which consists of $\Pp_0$-spaces and is  properly contained in the class of $\aleph_0$-spaces. A regular topological space $X$ is a {\em $\Pp_0$-space} if and only if $X$ has a countable $cp$-network. A family $\mathcal N$ of subsets of a topological space $X$ is called a {\em $cp$-network} if for any point $x\in X$, each neighborhood $O_x\subset X$ of $x$ and every subset $A\subset X$ containing $x$ in its closure $\bar A$ there exists a set $N\in\Nn$ such that $x\in N\subset O_x$ and moreover $N\cap A$ is infinite if $A$ accumulates at $x$ (see \cite{Banakh,GK-GMS1}). So each space $X\in \cc[\M_0]$ is even a $\Pp_0$-space. However, there exists a  $\Pp_0$-space $X$ which does not belong to the class $\cc[\M]$ (see Example \ref{exa-Pyt0-C[M]}).

Following \cite{GK-GMS1}, we define a topological space $Y$ to be a {\em $\Pp$-space} if $Y$ is regular and  possesses a $\sigma$-locally finite $cp$-network. According to \cite{GK-GMS1}, this class is closed under taking subspaces, topological sums and countable Tychonoff products. Each  $\Pp$-space is an  $\aleph$-space. A topological space $X$ is a $\Pp_0$-space if and only if $X$ is a Lindel\"{o}f $\Pp$-space \cite{GK-GMS1}.  In Section \ref{s:saak} we prove the following theorem which gives a partial answer to Question 6.2 in \cite{GK-GMS1}.

\begin{mytheorem} \label{c-Pyt-A-Met}
For any $\aleph_0$-space $X$ and each metrizable space $Y$ the function space $\CC(X,Y)$ is a paracompact $\Pp$-space.
\end{mytheorem}

The above discussion and Theorems \ref{t-Main-C(M)} and \ref{c-Pyt-A-Met} immediately imply

\begin{mycorollary} \label{c-Pyt-C[M]}
The class $\cc[\M_0]$ is contained in the class of $\Pp_0$-spaces, and the class $\cc[\M]$ is contained in the class of paracompact $\Pp$-spaces.
\end{mycorollary}

Since any $\Pp$-space has countable tightness \cite{GK-GMS1}, we obtain that also all spaces in the class $\cc[\M]$ have countable tightness.

Another topological property which holds for every space in the class $\cc[\M]$ can be obtained by modifying the $P$-space property.  Recall that a point $x$ of a topological space $X$ is called a {\em $P$-point} if for any countable family $\U$ of neighborhoods of $x$ the intersection $\bigcap\U$ is a neighborhood of $x$; $X$ is called a {\em $P$-space} if each $x\in X$ is a $P$-point. A point $x$ of a topological space $X$ is defined to be a {\em $P^{\w_1}_\w$-point} if for any uncountable family $\mathcal U$ of neighborhoods of $x$ there is a countably infinite subfamily $\mathcal V\subset\U$ whose intersection $\bigcap\V$ is a neighborhood of $x$. A topological space $X$ is called a {\em $P^{\w_1}_\w$-space} if each its point is a $P^{\w_1}_\w$-point. In Section \ref{s:p} we prove the following theorem.
\begin{mytheorem}\label{c:smP}
For any separable metrizable space $X$ and any metrizable space $Y$,  the function space $\CC(X,Y)$ is a $P^{\w_1}_\w$-space.
\end{mytheorem}
Now Theorems \ref{t-Main-C(M)} and \ref{c:smP} imply
\begin{mycorollary} \label{c-Pww-C[M]}
Any space $X$ in the class $\cc[\M]$ is a $P^{\w_1}_\w$-space.
\end{mycorollary}

Corollaries \ref{c-Pyt-C[M]} and \ref{c-Pww-C[M]} give us ``upper'' bounds for the class $\cc[\M]$. To obtain a  ``lower'' bound for $\cc[\M]$, in Section~\ref{s:ascoli} we introduce and study the class of Ascoli spaces. Recall that, for topological spaces $X$ and $Y$, a subset $\KK$ of $\CC(X,Y)$ is {\em evenly continuous} if the map $(f,x)\mapsto f(x)$ is continuous as a map from $\KK\times X$ to $Y$, i.e. for any $f\in\KK$, $x\in X$ and neighborhood $O_{f(x)}\subset Y$  of $f(x)$ there exist neighborhoods $U_f\subset \K$ of $f$ and $O_x\subset X$ of $x$ such that $U_f(O_x):=\{g(y):g\in U_f,\;y\in O_x\}\subset O_{f(x)}$.
A topological space $X$ is called an {\em Ascoli space} if each compact subset $\KK$ of $\CC(X)$ is evenly continuous. By Ascoli's theorem \cite[3.4.20]{Eng}, each $k$-space, and hence each sequential space, is Ascoli.
\begin{mytheorem} \label{t-Ascoli}
Any Ascoli $\aleph_0$-space belongs to  the class $\cc[\M_0]$.
\end{mytheorem}

If an Ascoli space $Z$ embeds into some $\CC(X,Y)$ with $X,Y\in\M_0$, then $Z$ is an $\aleph_0$-space by \cite{Mich}. Conversely, if $Z$ is an $\aleph_0$-space, then $Z\in\cc[\M_0]$ by Theorem \ref{t-Ascoli}, and hence $Z$ embeds into some $\CC(X,Y)$ with $X,Y\in\M_0$ by Theorem \ref{t-Main-C(M)}. So we obtain
\begin{mycorollary} \label{c:Ascoli-embed}
An Ascoli space $Z$ embeds into a function space $\CC(X,Y)$ with  $X,Y\in\M_0$ if and only if $Z$ is an $\aleph_0$-space.
\end{mycorollary}

Spaces that belong to the classes $\cc[\M_0]$ and $\cc[\M]$ will be called {\em \mbox{$\CC[\M_0]$-spaces}} and {\em $\CC[\M]$-spaces}, respectively. We summarize the obtained and known results in the following diagram.

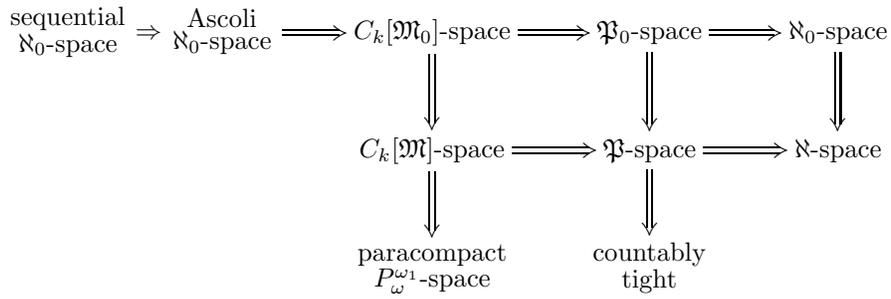
\[
\xymatrix{
{\mbox{sequential}\atop\mbox{$\aleph_0$-space}}  \Rightarrow {\mbox{Ascoli}\atop\mbox{$\aleph_0$-space}}\ar@{=>}[r] &\mbox{$\CC[\M_0]$-space}\ar@{=>}[d]\ar@{=>}[r]
&\mbox{$\Pp_0$-space}\ar@{=>}[d]\ar@{=>}[r]
&\mbox{$\aleph_0$-space}\ar@{=>}[d]\\
 &\mbox{$\CC[\M]$-space}\ar@{=>}[d]\ar@{=>}[r]&\mbox{$\Pp$-space}\ar@{=>}[r]\ar@{=>}[d]
&\mbox{$\aleph$-space}\\
 &\mbox{paracompact}\atop\mbox{$P^{\w_1}_\w$-space}&\mbox{countably} \atop \mbox{tight}
}
\]
\smallskip

Counterexamples constructed in the last section show that none of these implications can be reversed.

By analogy with the $\CC$-stable closure $\cc[\X]$  of a class $\X$ of topological spaces and being motivated also by the theory of Generalized Metric Spaces, in the next our paper \cite{BG-2} we introduce and characterize some natural types of the $C_p$-stable closures  of the class $\M_0$.




\section{Characterizations of the classes $\cc[\M_0]$ and $\cc[\M]$}\label{sec-Charac-C(M)}



A topological space $X$ {\em embeds} into a topological space $Y$ if there exists a topological embedding $e:X\hookrightarrow Y$. Recall that a map \mbox{$f:X\to Y$} between topological spaces is called {\em compact-covering} if for each compact subset $K\subset Y$ there is a compact subset $C\subset X$ such that $K= f(C)$.

\begin{myprop} \label{p:Ascoli}
If $\X$ is a non-empty class of topological spaces, then
\begin{enumerate}
\item[{\rm (1)}] $[\X]$ contains all zero-dimensional separable metrizable spaces;
\item[{\rm (2)}] $[\X]$ contains all separable metrizable spaces, provided that $\II\in \X$;
\item[{\rm (3)}] $\CC(Z,X) \in \cc[\X]$ for every $\aleph_0$-space $Z$ and any $X\in\X$.
\end{enumerate}
\end{myprop}

\begin{proof}
(1) The class $\X$ is not empty and hence contains some topological space. Then its stable closure $[\mathcal X]$ contains the empty topological space and its 0-th power $\emptyset^0$ which is a singleton (this follows from the assumption that $[\X]$ is closed under taking countable Tychonoff products).
Being closed under countable topological sums, the class $[\X]$ contains all countable discrete spaces, in particular, the doubleton $2=\{0,1\}$. By the countable productivity,  $[\X]$ contains the Cantor cube $2^\w$ and all its subspaces. Since each zero-dimensional separable metrizable  space embeds into the Cantor cube \cite[7.8]{Ke}, the class $[\X]$ contains all zero-dimensional separable metrizable spaces.
\smallskip

(2) If $\II\in \X$, then $\II^\NN \in [\X]$ and hence $[\M_0]\subseteq [\X]$ by \cite[4.2.10]{Eng}.
\smallskip

(3) The space $Z$, being an $\aleph_0$-space, is the image of a separable metrizable space $M$ under a compact-covering map $\xi:M\to Z$ (see \cite{Mich}). Since every separable metrizable space is the image of a zero-dimensional separable metrizable space under a perfect (and hence compact-covering) map, without loss of generality we can suppose that the space $M$ is zero-dimensional and hence belongs to the class $[\X]$ by the first statement. The  $\CC$-stability of $\cc[\X]$ guarantees that $\CC(M,X)\in\cc[\X]$. The compact-covering property of the map $\xi$ implies that the dual map $\xi^*:\CC(Z,X)\to \CC(M,X)$, $\xi^*:f\mapsto f\circ\xi$, is a topological embedding.  Thus also $\CC(Z,X)$ belongs to $\cc[\X]$.
\end{proof}

Theorem \ref{t-Main-C(M)} announced in the introduction is a partial case of the following two theorems.

\begin{mytheorem}\label{t:CMemb}
The class $\cc[\M]$ coincides with the class $\cc(\M_0,\M)$.
\end{mytheorem}

\begin{proof}
The definition of the class $\cc[\M]$ guarantees that $\cc(\M_0,\M)\subset \cc[\M]$. By the minimality of $\cc[\M]$, the reverse inclusion will be proved as soon as we will check that the class $\cc(\M_0,\M)$ is  closed under taking countable topological sums, countable Tychonoff products and taking function spaces with Lindel\"{o}f domain.
\smallskip

To see that the class $\cc(\M_0,\M)$ is closed under taking countable topological sums, we need to prove that for any non-empty spaces $X_n\in\M_0$ and $Y_n\in\M$, $n\in\w$, the topological sum $\bigoplus_{n\in\w}\CC(X_n,Y_n)$ embeds into the function space $\CC(X,Y)$ for some spaces $X\in\M_0$ and $Y\in\M$. Fix a singleton $\ast$. Now we consider the topological sums $X=\bigoplus_{n\in\w}X_n$ and $Y=\{\ast\}\oplus\bigoplus_{n\in\w}Y_n$ and the topological embedding
\[
e:\textstyle{\bigoplus\limits_{n\in\w}}\CC(X_n,Y_n)\hookrightarrow \CC(X,Y)
\]
assigning to each function $f\in \CC(X_n,Y_n)$, $n\in\w$, the function $\tilde f\in \CC(X,Y)$ defined by
\[
\tilde f(x)=\begin{cases}f(x), &\mbox{if $x\in X_n$};\\
\ast , & \mbox{if $x\in X_m$ and $m\ne n$.}
\end{cases}
\]
Since $X\in\M_0$ and $Y\in\M$, we conclude that $\CC(X,Y)\in \cc(\M_0,\M)$, and hence $\bigoplus_{n\in\w}\CC(X_n,Y_n)\in \cc(\M_0,\M)$.
\smallskip

Next we prove that for any Lindel\"of space $X\in \cc(\M_0,\M)$ and any space $Y\in \cc(\M_0,\M)$ the function space $\CC(X,Y)$ belongs to the class $\cc(\M_0,\M)$. The Foged theorem~\cite{foged} implies that the space $X$ is an $\aleph$-space. Being Lindel\"of, $X$ is an $\aleph_0$-space. By \cite{Mich}, the space $X$ is the image of a separable metrizable space $M$ under a compact-covering map. Hence, by Lemma 1 of \cite{OMe2}, the function space $\CC(X,Y)$ embeds into the function space $\CC(M,Y)$. So, it is enough to prove that the function space $\CC(M,Y)$ belongs to the class $\cc(\M_0,\M)$. By definition, the space $Y$ embeds into the function space $\CC(A,B)$ for some spaces $A\in\M_0$ and $B\in\M$. Consequently, the function space $\CC(M,Y)$ embeds into the function space $\CC(M,\CC(A,B))$. By Theorem 3.4.9 of \cite{Eng}, the map $\CC(M,\CC(A,B))\to \CC(M\times A,B)$ assigning to each function $f:M\to \CC(A,B)$ the function $\tilde f:M\times A\to B$, $\tilde f:(m,a)\mapsto f(m)(a)$, is a homeomorphism. Since $M\times A$ is a separable metrizable space, the function space $\CC(M\times A,B)$ belongs to the class $\cc(\M_0,\M)$, and hence  the spaces $\CC(M,\CC(A,B))$, $\CC(M,Y)$ and $\CC(X,Y)$ also belong to $\cc(\M_0,\M)$.
\smallskip

Finally we show that the class $\cc(\M_0,\M)$ is closed under taking countable Tychonoff products. Fix any spaces $X_n\in \cc(\M_0,\M)$, $n\in\w$. As the class $\cc(\M_0,\M)$ is closed under taking countable topological sums, the topological sum $X=\bigoplus_{n\in\w}X_n$ belongs to the class $\cc(\M_0,\M)$. Since the class $\cc(\M_0,\M)$ is closed also under taking function spaces with Lindel\"of domain, the function space $\CC(\w,X)$ belongs to the class $\cc(\M_0,\M)$. Taking into account that $\prod_{n\in\w}X_n\subset X^\w=\CC(\w,X)\in \cc(\M_0,\M)$, we conclude that $\prod_{n\in\w}X_n\in \cc(\M_0,\M)$.
\end{proof}

Below we characterize the class $\cc[\M_0]$.

\begin{mytheorem}\label{t:CM0}
For a topological space $X$ the following conditions are equivalent:
\begin{enumerate}
\item[{\rm (1)}] $X$ belongs to the class $\cc[\M_0]$;
\item[{\rm (2)}] $X$ is Lindel\"of and belongs to the class $\cc[\M]$;
\item[{\rm (3)}] $X$ embeds into the function space $\CC(Z,Y)$ for some spaces $Z,Y\in\M_0$;
\item[{\rm (4)}] $X$ embeds into the function space $\CC(Z)$ for some zero-dimensional space $Z\in\M_0$.
\end{enumerate}
\end{mytheorem}

\begin{proof}
(1)$\Ra$(2): As $\cc[\M_0]\subset \cc[\M]$, we have also $X\in \cc[\M]$. The space $X$, being an $\aleph_0$-space, is Lindel\"of by \cite{Mich}.
\smallskip

(2)$\Ra$(3): Assume that $X$ is a Lindel\"of space from the class $\cc[\M]$. By Theorem~\ref{t:CMemb}, $X$ embeds into the function space $\CC(Z,Y)$ for some spaces \mbox{$Z\in\M_0$} and $Y\in\M$. The Foged theorem~\cite{foged} implies that the space $X$ is an $\aleph$-space. Being Lindel\"of, $X$ is an $\aleph_0$-space (see \cite{GK2}). By \cite{Mich}, $X$ is the image of a separable metrizable space $A$ under a compact-covering continuous map $\xi:A\to X\subset \CC(Z,Y)$. Consider the map $\varphi:A\times Z\to Y$ assigning to each pair  $(a,z)\in A\times Z$ the value $\xi(a)(z)$ of the function $\xi(a)$ at the point $z$. We claim that this map is continuous at each point $(a,z)\in A\times Z$. Take any sequence $\big((a_k,z_k)\big)_{k\in\w}$ in $A\times Z$ converging to $(a,z)$ in the  separable metrizable space $A\times Z$, and any neighborhood $O_{\varphi(a,z)}\subset Y$ of the point $\varphi(a,z)$. For every $k\in\w$ consider the function $f_k :=\xi(a_k)\in X\subset \CC(Z,Y)$ and observe that the sequence $\{ f_k\}_{k\in\w}$ converges to the function $f:=\xi(a)$. By Ascoli's theorem  \cite[3.4.20]{Eng}, the compact subset $\K :=\{f\}\cup\{f_k\}_{k\in\w}\subset \CC(Z,Y)$ is evenly continuous, that allows us to find a neighborhood $U_f$ of $f$ in $\K$ and a neighborhood $O_z$ of $z$ such that $U_f(O_z)\subset O_{\varphi(a,z)}$. By the continuity of the map $\xi$, there is a number $k_0\in\w$ such that $\xi(a_k)\in U_f$ for all $k\ge k_0$. Then for any $k\ge k_0$ with $z_k\in O_z$ we get $\varphi(a_k,z_k)=f_k(z_k)\in U_f(O_z)\subset O_{\varphi(a,z)}$.
So, the map $\varphi:A\times Z\to Y$ is continuous, and hence its image $Y_0:=\varphi(A\times Z)\subset Y$ is separable. Then $X\subset \CC(Z,Y_0)\subset \CC(Z,Y)$, where the spaces $Z$ and $Y_0$ are  separable and metrizable.
\smallskip

(3)$\Ra$(4): Assume that $X$ embeds into the function space $\CC(A,B)$ for some separable metrizable spaces $A$ and $B$. Since the separable metrizable space $B$ embeds into the countable product $\IR^\w$, the function space $\CC(A,B)$ embeds into the function space $\CC(A,\IR^\w)$, which is homeomorphic to the function space $\CC(A\times \w,\IR)=\CC(A\times\w)$. The separable metrizable space $A\times\w$ can be written as the image of a zero-dimensional separable metrizable space $Z$ under a compact-covering map $\eta:Z\to A\times\w$. Then the dual map $\eta^*:\CC(A\times\w)\to \CC(Z)$, $\eta^*:f\mapsto f\circ\eta$, is a topological embedding (see Lemma~1 in \cite{OMe2}). Thus
$$
X\hookrightarrow \CC(A,B)\hookrightarrow  \CC(A,\IR^\w)=\CC(A\times\w)\hookrightarrow \CC(Z).
$$

(4)$\Ra$(1): Assume that $X$ embeds into the function space $\CC(Z)$ for some zero-dimensional separable metrizable space $Z$. Since $Z,\IR\in\M_0$, the function space $\CC(Z)=\CC(Z,\IR)$ as well as its subspace $X$ belong to the class $\cc[\M_0]$.
\end{proof}


\section{Proof of Theorem \ref{c-Pyt-A-Met}}\label{s:saak}


In this section we prove a more general result than Theorem \ref{c-Pyt-A-Met}. For this purpose we need some definitions.

A family $\I$ of compact subsets of a topological space $X$ is called an {\em ideal of compact sets} if $\bigcup\I=X$ and for any sets $A,B\in\I$ and any compact subset $K\subset X$ we get $A\cup B\in\I$ and $A\cap K\in\I$, i.e. if $\I$ covers $X$ and is closed under taking finite unions and closed subspaces.

For any two topological spaces $X$ and $Y$, each ideal $\I$ of compact subsets of $X$ determines the {\em $\I$-open topology} on the space $C(X,Y)$ of continuous functions from $X$ to $Y$. A subbase of this topology consists of the sets
\[
[K;U]=\{f\in C(X,Y):f(K)\subset U\},
\]
where $K\in\I$ and $U$ is an open subset of $Y$. The space  $C(X,Y)$ endowed with the $\I$-open topology will be denoted by $C_\I(X,Y)$. So, $\CC(X,Y)=C_\I(X,Y)$ for the ideal $\I$ of all compact subsets of $X$. For the ideal $\I$ of all finite subsets of $X$, the function space $C_\I(X,Y)$ is denoted by $C_p(X,Y)$.

We are interested in detecting $\Pp$-spaces among the function spaces $C_\I(X,Y)$.

\begin{mydefinition}
An ideal  $\I$  of compact subsets of a topological space $X$ is called {\em discretely-complete} if for any compact subsets $A,B\subset X$ with countable discrete difference $B\setminus A$, the inclusion $A\in\I$ implies $B\in\I$.
\end{mydefinition}

Since the ideal of all compact subsets of a topological space is trivially discretely-complete, the following theorem implies Theorem \ref{c-Pyt-A-Met}.
\begin{mytheorem}\label{saak}
For any discretely-complete ideal $\I$ of compact subsets of an \mbox{$\aleph_0$-space} $X$ and any metrizable space $Y$ the function space $C_\I(X,Y)$ is a paracompact  $\Pp$-space. If $Y$ is separable, then $C_\I(X,Y)$ is a  $\Pp_0$-space.
\end{mytheorem}

\begin{proof}
For the $\aleph_0$-space $X$ fix a countable $k$-network  $\KK$, which is closed under taking finite unions and finite intersections. For the metrizable space $Y$ fix a  $\sigma$-locally finite base of the topology  $\DD =\bigcup_{j\in\w} \DD_j$ (which exists by the Nagata-Smirnov metrization theorem \cite[4.4.7]{Eng}). We claim that the family
\[
[\kern-2pt[ \KK; \DD ]\kern-1.5pt]=\big\{[K_1;D_1]\cap \dots\cap [K_n;D_n]:K_1,\dots,K_n\in\K,\;D_1,\dots,D_n\in\mathcal D\big\}
\]
is a $\sigma$-locally finite $cp$-network in $C_\I(X,Y)$.

To see that the family $[\kern-1.5pt[ \KK; \DD ]\kern-1.5pt]$ is $\sigma$-locally finite in $C_\I(X,Y)$, write it as
\[
[\kern-1.5pt[ \KK; \DD ]\kern-1.5pt] = \bigcup_{n\in\IN}\bigcup_{K_1,\dots,K_n\in\K} \bigcup_{(l_1,\dots,l_n)\in\w^n} [\kern-1.5pt[ K_1; \DD_{l_1} ]\kern-1.5pt] \wedge\cdots\wedge [\kern-1.5pt[ K_n; \DD_{l_n} ]\kern-1.5pt] ,
\]
where
\[
[\kern-1.5pt[ K_1; \DD_{l_1} ]\kern-1.5pt] \wedge\cdots\wedge [\kern-1.5pt[ K_n; \DD_{l_n} ]\kern-1.5pt] := \big\{[K_1;D_1]\cap\dots\cap [K_n;D_n]:D_1\in \DD_{l_1},...,D_n\in\DD_{l_n}\big\}.
\]
We claim that for any $K\in\K$ and $l\in\w$ the family $[\kern-1.5pt[ K; \DD_{l} ]\kern-1.5pt]$ is locally finite in $C_\I(X,Y)$. Given any function $f\in C_\I(X,Y)$, choose a point $x\in K$. As $\bigcup\I=X$, the singleton $\{x\}$ belongs to the ideal $\I$. Since the family $\DD_l$ is locally finite in $Y$, the point $f(x)\in Y$ has a  neighborhood \mbox{$V\subset Y$} meeting only finitely many members of $\DD_l$. Then the open neighborhood $[\{ x\}; V]\subset C_\I(X,Y)$ of $f$ meets only those elements $[K;D]\in [\kern-1.5pt[ K; \DD_l]\kern-1.5pt]$ for which $D$ intersects $V$. By  the choice of $V$, the number of such elements is finite. So, the family $[\kern-1.5pt[K,\DD_l]\kern-1.5pt]$ is locally finite and therefore so is the family $[\kern-1.5pt[ K_1; \DD_{l_1} ]\kern-1.5pt] \wedge\cdots\wedge [\kern-1.5pt[ K_n; \DD_{l_n} ]\kern-1.5pt]$ for every sets $K_1,\dots,K_n\in\K$ and every numbers $l_1,\dots,l_n\in\w$.
\smallskip

Now we prove that the family $[\kern-1.5pt[\K;\DD]\kern-1.5pt]$ is a $cp$-network for $C_\I(X,Y)$. Fix any function $f\in C_\I(X,Y)$, an open neighborhood $O_f\subset C_\I(X,Y)$ of $f$ and a subset $A\subset C_\I(X,Y)$ containing $f$ in its closure. We lose no generality assuming that the neighborhood $O_f$ is of basic form $$O_f = [C_1;U_1]\cap\cdots\cap [C_n;U_n]$$for some compact sets $C_1,\dots,C_n\in\I$ and some open sets $U_1,\dots,U_n\in\DD$.

For every $i\in\{ 1,\dots, n\}$, consider the countable family
\[
\KK_i :=\{ K\in\KK : C_i \subset K\subset f^{-1}(U_i)\},
\]
and let $\KK_i = \{ K'_{i,j} \}_{j\in\w}$ be its enumeration. For every $j\in\NN$ we  put  $K_{i,j} := \bigcap_{k\leq j} K'_{i,k}$. It follows that the decreasing sequence $\{ K_{i,j}\}_{j\in\w}$ converges to $C_i$ in the sense that each open neighborhood of $C_i$ contains all but finitely many sets $K_{i,j}$.
Then the sets
\[
\FF_j := \bigcap_{i=1}^n [K_{i,j};U_{i}] \in [\kern-1.5pt[ \KK; \DD ]\kern-1.5pt],\;\;j\in\w,
\]
form an increasing sequence of sets in the function space $\CC(X,Y)$. We claim that
\[
\bigcap_{i=1}^n [C_i; U_i] = O_f = \bigcup_{j\in\w} \FF_j.
\]

Suppose for a contradiction that there exists a function $g\in  \bigcap_{i=1}^n [C_i; U_i]$ which does not belong to $\bigcup_{j\in\NN} \FF_j$. Then for every $j\in\w$ we can find an index $i_j\in \{1,\dots,n\}$ such that $g\not\in [K_{i_j,j};U_{i_j}]$. This means that $g(x_j)\not\in U_{i_j}$ for some point $x_j\in K_{i_j,j}$. By the Pigeonhole Principle, there is $m\in \{1,\dots,n\}$ such that the set $J_m := \{ j\in\NN : i_j =m\}$ is infinite. As the decreasing sequence $\{ K_{m,j} \}_{j\in J_m}$ converges to the compact set $C_m$, the set $C_m \cup \{ x_j\}_{j\in J_m}$ is compact.

Since any compact subset of the $\aleph_0$-space $X$ is metrizable, we can find an infinite subset $J'$ of $J_m$ such that the sequence  $\{ x_j\}_{j\in J'}$ converges to some point $x_0 \in C_m$. As $g$ is continuous, the sequence $\{ g(x_j)\}_{j\in J'}$ converges to the point $g(x_0) \in g(C_m) \subset U_m$,  and hence we can assume also that \mbox{$g(x_j)\in U_m$} for every $j\in J'$. But this contradicts the choice of the points $x_j$ and therefore proves the equality $O_f=\bigcup_{j\in\w}\F_j$.

It follows that $f\in\F_j\subset O_f$ for some $j\in\w$, which means that the family $[\kern-1.5pt[\mathcal K;\mathcal D]\kern-1.5pt]$ is a network in the function space $C_\I(X;Y)$. So without loss of generality we can assume that $f\in \F_j$ for every $j\in\w$.
\smallskip

Now, assuming that the set $A$ accumulates at $f$, we shall prove that for some $j\in\w$ the  intersection $\F_j\cap A$ is infinite. Replacing $A$ by $O_f\cap A$, we can assume that $A\subset O_f$.
Suppose for a contradiction that for every $j\in\w$ the intersection $A_j := \FF_j \cap A$ is finite. Then $A=A\cap O_f = \bigcup_{j\in\w} A_j$ is the countable union of the increasing sequence $\{ A_j\}_{j\in\w}$ of finite subsets of $C_\I(X,Y)$.

For every function $\alpha\in A\setminus A_0$ we denote by $j_\alpha$ the unique natural number such that $\alpha\in A_{j_\alpha +1} \setminus A_{j_\alpha} = A_{j_\alpha +1} \setminus \FF_{j_\alpha}$. Since $\alpha\not\in \FF_{j_\alpha} =\bigcap_{i=1}^n [K_{i,j_\alpha};U_{i}]$, there is an index $i_\alpha\in\{1,\dots,n\}$ such that $\alpha\not\in [K_{i_\alpha,j_\alpha};U_{i_\alpha}]$ and a point $x_\alpha\in K_{i_\alpha,j_\alpha}$ such that $\alpha(x_\alpha) \not\in U_{i_\alpha}$.

For every $i\in\{1,\dots,n\}$ consider the subsequence
\[
A(i):= \{ \alpha\in A\setminus A_0 : \ i_\alpha = i\}
\]
and observe that $A\setminus A_0 = \bigcup_{i=1}^n A(i)$. So, there exists $s \in\{1,\dots,n\}$ such that the set $A(s)$ is infinite and $f\in \overline{A(s)}$.

Taking into account that the sequence of sets $\{ K_{s,j}\}_{j\in\w}$ converges to the compact set $C_{s}$, we can find a finite subset $B$ of $A(s)$ such that $f(x_\alpha)\in U_{s}$ for every $\alpha\in A(s) \setminus B$. Put
\[
A' := A(s) \setminus B \quad \mbox{ and } \quad C' := C_{s} \cup \{ x_\alpha\}_{\alpha\in A'}
\]
and observe that the compact set $C'$ belongs to the ideal $\I$ by the discrete-completeness of $\I$.
Since $[C';U_{s}]$ is an open neighborhood of $f\in \overline{A'}$, there is a function $\alpha\in A'\cap [C';U_{s}]$, which is not possible as $\alpha(x_\alpha)\notin U_{s}$. Thus $\FF_j \cap A$ is infinite for some $j\in\w$. Therefore $[\kern-1.5pt[\K;\mathcal D]\kern-1.5pt]$ is a $cp$-network in $C_\I(X,Y)$ and hence $C_\I(X,Y)$ is a $\Pp$-space.
\smallskip

The paracompactness of the space $C_\I(X,Y)$ was implicitly proved in Lemmas 5 and 6 of \cite{OMe2}.
\smallskip

If the space $Y$ is separable, then the family $\DD$ is countable and so is the $cp$-network $[\kern-1.5pt[\K;\DD]\kern-1.5pt]$. This means that the function space $C_\I(X,Y)$ is a $\Pp_0$-space.
\end{proof}


\section{Detecting function spaces which are $P^{\w_1}_\w$-spaces}\label{s:p}


In this section we discuss a topological property that allows us to construct an example of a $\mathfrak P_0$-space, which does not belong to the class $\cc[\M]$ (see Example \ref{exa-Pyt0-C[M]}). This property is a modification of the $P$-space property, and its study was initiated by Greinecker and Ravsky in \\ {\tt http://math.stackexchange.com/questions/377038}.

\begin{mydefinition}
A point $x$ of a topological space $X$ is called a {\em $P^\kappa_\lambda$-point} for cardinals $\lambda\le\kappa$ if for any family $(U_\alpha)_{\alpha\in\kappa}$ of neighborhoods of $x$ there is a subset $\Lambda\subset\kappa$ of cardinality $|\Lambda|=\lambda$ such that $\bigcap_{\alpha\in\Lambda}U_\alpha$ is a neighborhood of $x$. The space $X$ is called a {\em $P^\kappa_\lambda$-space} if each $x\in X$ is a $P^\kappa_\lambda$-point.
\end{mydefinition}

The next proposition implies that a topological space $X$ is a $P$-space if and only if it is a $P^\w_\w$-space.

\begin{myprop}\label{p:P}
A point $x$ of a topological space $X$ is a $P$-point if and only if it is a $P^\w_\w$-point.
\end{myprop}

\begin{proof}
The ``only if'' part of this statement if trivial. To prove the ``if'' part, assume that $x$ is a $P^\w_\w$-point. To prove that $x$ is a $P$-point, take any sequence $(U_n)_{n\in\w}$ of neighborhoods of $x$. For every $n\in\w$, put $V_n=\bigcap_{i\le n}U_i$. Since $x$ is a $P^\w_\w$-point, for the sequence of neighborhoods $(V_n)_{n\in\w}$ of $x$ there is an infinite subset $\Lambda\subset\w$ such that $\bigcap_{n\in\Lambda}V_n$ is a neighborhood of $x$. Since the sequence $(V_n)_{n\in\w}$ is decreasing, we conclude that $\bigcap_{n\in\w}U_n=\bigcap_{n\in\w}V_n=\bigcap_{n\in\Lambda}V_n$ is a neighborhood of $x$. So, $x$ is a $P$-point in $X$.
\end{proof}

For a point $x$ of a topological space $X$, denote by $\chi(x,X)$ the {\em character} of $X$ at $x$, i.e., the smallest cardinality of a neighborhood base at $x$.

\begin{myprop}\label{p:chiP}
A point $x$ of a topological space $X$ is a $P^\kappa_\lambda$-point for any cardinals $\lambda <\kappa$ with $\kappa>\chi(x,X)$.
\end{myprop}

\begin{proof}
Fix a neighborhood base $\mathcal B_x$ at $x$ of cardinality $|\mathcal B_x|=\chi(x,X)$. To show that $x$ is a $P^\kappa_\lambda$-point for cardinals $\lambda <\kappa>\chi(x,X)$, take any family of neighborhoods $(U_\alpha)_{\alpha\in\kappa}$ of $x$. For each $\alpha\in\kappa$ find a basic set $B_\alpha\in\mathcal B_x$ which is contained in the neighborhood $U_\alpha$. Since $|\mathcal B_x|=\chi(x,X)<\kappa$, by the Pigeonhole Principle, there is $B\in\mathcal B_x$ such that the set $\Lambda=\{\alpha\in\kappa:B_\alpha=B\}$ has cardinality $|\Lambda|>\lambda$. Then $\bigcap_{\alpha\in\Lambda}U_\alpha\supset B$ is a neighborhood of $x$.
\end{proof}
By Proposition~\ref{p:chiP}, each first countable space is a $P^{\kappa}_{\w}$-space for any uncountable cardinal $\kappa$.

To detect function spaces $C_\I(X,Y)$ which are $P^\kappa_\lambda$-spaces, let us introduce the corresponding property for ideals of compact sets.

\begin{mydefinition}
An ideal $\I$ of compact subsets of a topological space $X$ is called a {\em $P^\kappa_\lambda$-ideal} for cardinals $\lambda\leq\kappa$ if for any family of compact subsets \mbox{$\{K_\alpha\}_{\alpha\in\kappa}\subset\I$} there exists a subset $\Lambda\subset\kappa$ of cardinality $|\Lambda|=\lambda$ such that the union $\bigcup_{\alpha\in\Lambda}K_\alpha$ is contained in some compact set $K\in\I$.
\end{mydefinition}

\begin{mytheorem}\label{t:P} Let $\lambda\le\kappa$ be two cardinals with $\cf(\kappa)>\w$. For an ideal $\I$ of compact subsets of a Tychonoff space $X$ the following conditions are equivalent:
\begin{enumerate}
\item[{\rm (1)}] $\I$ is a $P^\kappa_\lambda$-ideal;
\item[{\rm (2)}] for every metrizable space $Y$ the function space $C_\I(X,Y)$ is a $P^\kappa_\lambda$-space;
\item[{\rm (3)}] the function space $C_\I(X)$ is a $P^\kappa_\lambda$-space.
\end{enumerate}
\end{mytheorem}

\begin{proof} (1)$\Rightarrow$(2): Assume that $\I$ is a $P^\kappa_\lambda$-ideal. We need to show that for any metrizable space $Y$ the function space $C_\I(X,Y)$ is a $P^\kappa_\lambda$-space. Since each metrizable space embeds into a Banach space, we can assume that $Y$ is a Banach space endowed with a norm $\|\cdot\|$. Then $C_\I(X,Y)$ is a linear topological space. So it suffices to check that the constant zero function \mbox{$\mathbf{0}:X\to \{0\}\subset Y$} is a $P^\kappa_\lambda$-point of the function space $C_\I(X,Y)$. Fix a family $(U_\alpha)_{\alpha\in\kappa}$ of neighborhoods of $\mathbf{0}$ in $C_\I(X,Y)$. By the definition of the topology of the function space $C_\I(X,Y)$, for every $\alpha\in\kappa$ there is a compact set $K_\alpha\in\I$ and a positive real number $\e_\alpha$ such that the neighborhood $U_\alpha$ contains the basic neighborhood
\[
[K_\alpha;\e_\alpha]:=\{f\in C_\I(X,Y):\sup_{x\in K_\alpha}\|f(x)\|<\e_\alpha\}.
\]
Since the cardinal $\kappa$ has uncountable cofinality, there is a positive real number $\e$ such that the set \mbox{$\Omega=\{\alpha\in\kappa:\e_\alpha\ge\e\}$} has cardinality $|\Omega|=\kappa$. Since $\I$ is a $P^\kappa_\lambda$-ideal, for the family of compact sets $\{K_\alpha\}_{\alpha\in\Omega}\subset\I$
there is a subset $\Lambda\subset\Omega$ of cardinality $|\Lambda|=\lambda$ such that the union $\bigcup_{\alpha\in\Lambda}K_\alpha$ is contained in some compact set $K\in\I$. Then the intersection $\bigcap_{\alpha\in\Lambda}U_\alpha$ contains the neighborhood $[K;\e]$ of $\mathbf{0}$, which means that $\mathbf{0}$ is a $P^\kappa_\lambda$-point of $C_\I(X,Y)$.
\smallskip

The implication (2)$\Rightarrow$(3) is trivial. To prove that $(3)\Rightarrow(1)$, assume that the function space $C_\I(X)$ is a $P^\kappa_\lambda$-space. We have to prove that $\I$ is a $P^\kappa_\lambda$-ideal. Fix any family of compact sets $\{K_\alpha\}_{\alpha\in\kappa}\subset\I$ and for every $\alpha\in\kappa$ consider the neighborhood
\[
U_\alpha=\{f\in C_\I(X):\sup_{x\in K_\alpha}|f(x)|<1\}
\]
of the constant zero function $\mathbf{0}\in C_\I(X)$. Since $C_\I(X)$ is a $P^\kappa_\lambda$-space, there exists a subset $\Lambda\subset\kappa$ of cardinality $|\Lambda|=\lambda$ such that the intersection $\bigcap_{\alpha\in\Lambda}U_\alpha$ is a neighborhood of $\mathbf{0}$. Then this intersection contains a basic neighborhood $[K;\e]$ for some compact set $K\in\I$ and some $\e>0$. We claim that \mbox{$\bigcup_{\alpha\in\Lambda}K_\alpha\subset K$}. Indeed, assuming the converse we can find $\alpha\in\Lambda$ and a point $x\in K_\alpha\setminus K$. Since the space $X$ is Tychonoff, there exists a continuous function $f:X\to\IR$ such that $f(K)\subset\{0\}$ and $f(x)=1$. Then $f\in [K;\e]\not\subset U_\alpha$ that contradicts the choice of $[K;\e]$.
\end{proof}

Since the ideal of finite subsets of an uncountable topological space fails to be a $P^{\w_1}_\w$-ideal, Theorem~\ref{t:P} implies:

\begin{mycorollary}
For any uncountable space Tychonoff space $X$ the function space $C_p(X)$ fails to be a $P^{\w_1}_\w$-space.
\end{mycorollary}

An ideal $\I$ of compact subsets of a topological space $X$ is called a {\em \mbox{$\sigma$-ideal}} if each compact subset $K\subset X$, which can be written as the countable union \mbox{$K=\bigcup_{n\in\w}K_n$} of compact sets $K_n\in\I$, $n\in\w$, belong to the ideal $\I$. For example, for any infinite cardinal $\kappa$ the ideal of all compact subsets of cardinality $\leq \kappa$ in a topological space $X$ is a $\sigma$-ideal.

Now we see that Theorem \ref{c:smP} is a partial case of the following theorem.
\begin{mytheorem}\label{t:sigmaP}
Any $\sigma$-ideal $\I$ of compact subsets of a  separable metrizable space $X$ is a $P^{\w_1}_\w$-ideal. Consequently, for every metrizable space $Y$ the function space $C_\I(X,Y)$ is a $P^{\w_1}_\w$-space.
\end{mytheorem}

\begin{proof} To show that $\I$ is a $P^{\w_1}_\w$-ideal, fix any indexed family $\K=\{K_\alpha\}_{\alpha<\w_1}$ of non-empty compact sets in $\I$.  If the set $\K$ is countable, then for some compact set $K\in\K$ the set \mbox{$\Lambda=\{\alpha\in\kappa:K_\alpha=K\}$} is infinite and the union \mbox{$\bigcup_{\alpha\in\lambda}K_\alpha =K$ } belongs to $\I$. So, we assume that $\K$ is uncountable. The set $\K$ can be considered as an uncountable subset of the hyperspace $\exp(X)$ of all non-empty compact sets endowed with the Vietoris topology. It is well-known that the hyperspace $\exp(X)$ of the separable metrizable space $X$ is separable and metrizable (by the Hausdorff metric). Since separable metrizable spaces do not contain uncountable discrete subspaces, the uncountable set $\K\subset\exp(X)$ is not discrete and hence contains a non-trivial convergent sequence with the limit point. So we can find a sequence of pairwise distinct countable ordinals $(\alpha_n)_{n\in\w}$ such that the sequence $(K_{\alpha_n})_{n\in\w}$ converges to some compact set $K_\infty\in\K$ in the hyperspace $\exp(X)$. It follows that the union
\[
K=K_\infty\cup\bigcup_{n\in\w}K_{\alpha_n}
\]
is a compact set in $X$. Since $\I$ is a $\sigma$-ideal, the compact set $K$ belongs to the ideal $\I$ and contains the union $\bigcup_{n\in\w}K_{\alpha_n}$ witnessing that $\I$ is a $P^{\w_1}_\w$-ideal.

By Theorem~\ref{t:P}, for every metrizable space $Y$ the function space $C_\I(X,Y)$ is a $P^{\w_1}_\w$-space.
\end{proof}


\section{Ascoli spaces}\label{s:ascoli}


The results of this section are motivated by the problem of detecting spaces that necessarily belong to the $\CC$-stable closure $\cc[\X]$ of any non-empty  class of topological spaces $\X$. To detect spaces that embed into some nice function spaces we shall exploit the construction of the canonical map.

For any topological spaces $X$ and $Y$ the canonical map
\[
\delta:X\to \CC(\CC(X,Y),Y)
\]
assigns to each point $x\in X$ the {\em $Y$-valued Dirac measure} $\delta_x:\CC(X,Y)\to Y$ concentrated at $x$. The $Y$-valued Dirac measure $\delta_x$ assigns to each function \mbox{$f\in \CC(X,Y)$} its value $f(x)$ at $x$; so $\delta_x(f)=f(x)$. The function $\delta_x:\CC(X,Y)\to Y$ is continuous at each function $f\in \CC(X,Y)$ since for any open neighborhood $V\subset Y$ of $\delta_x(f)=f(x)$, the set $[\{x\};V]$ is an open neighborhood of $f$ in $\CC(X,Y)$ with $\delta_x\big([\{x\};V]\big)\subset V$. This shows that the canonical map $\delta:X\to \CC(\CC(X,Y),Y)$ is well-defined.

We are interested in detecting topological spaces $X$ and $Y$ for which the canonical map $\delta:X\to \CC(\CC(X,Y),Y)$ is a topological embedding. We start with finding conditions implying the continuity of the canonical map $\delta$.

Recall that a topological space $X$ is called
\begin{itemize}
\item[$\bullet$] {\em sequential} if for each non-closed subset $A\subset X$ there is a sequence $\{a_n\}_{n\in\w}\subset A$ converging to some point $a\in \bar A\setminus A$;
\item[$\bullet$] a {\em $k$-space} if for each non-closed subset $A\subset X$ there is a compact subset $K\subset X$ such that $A\cap K$ is not closed in $K$;
\item[$\bullet$] a {\em $k_\IR$-space} if  a real-valued function $f$ on $X$ is continuous if and only if its restriction $f|_K$ to any compact subset $K$ of $X$ is continuous.
\end{itemize}

In the next definition we generalize the notions of a $k_\mathbb{R}$-space and of the Ascoli space defined in Introduction.

\begin{mydefinition}
Let $Y$ be a topological space. A topological space $X$ is called
\begin{itemize}
\item[$\bullet$] a {\em $k_Y$-space} if for any discontinuous function $f:X\to Y$ there is a compact subset $K\subset X$ such that the restriction $f|_K$ is discontinuous;
\item[$\bullet$] a {\em $Y$-Ascoli space} if each compact subset $\K\subset \CC(X,Y)$ is evenly continuous.
\end{itemize}
\end{mydefinition}
So, $X$ is an  Ascoli space if and only if $X$ is $\IR$-Ascoli.   By Ascoli's theorem \mbox{\cite[3.4.20]{Eng}}, each $k$-space is $Y$-Ascoli for any regular space $Y$. On the other hand, Noble \cite{Noble} proved that each $k_\IR$-space is $Y$-Ascoli for any Tychonoff space $Y$. By Corollary \ref{l:AscoliCat} proved below, each Ascoli space is $2$-Ascoli, where  $2=\{0,1\}$ is the doubleton endowed with the discrete topology.
Therefore we have the following implications:
\[
\mbox{sequential $\Rightarrow$ $k$-space $\Rightarrow$ $k_\IR$-space $\Rightarrow$ Ascoli  $\Rightarrow$ 2-Ascoli.}
\]

Now we prove some elementary properties of $Y$-Ascoli spaces. Let us recall that a subspace $Z$ of a topological space $X$ is a {\em retract of $X$} if there is a continuous map $r:X\to Z$ such that $r(z)=z$ for all $z\in Z$.

\begin{myprop} \label{p:AscoliCat}
Let $Y$ be a topological space.
\begin{enumerate}
\item[{\rm (1)}] If $X$ is a $Y$-Ascoli space, then each retract $Z$ of $X$ is $Y$-Ascoli.
\item[{\rm (2)}] For any family $\{  X_i\}_{i\in I}$ of $Y$-Ascoli spaces the topological sum $X:=\bigoplus_{i\in I} X_i$ is $Y$-Ascoli.
\item[{\rm (3)}] Each $Y$-Ascoli space $X$ is $Y^\kappa$-Ascoli for every cardinal $\kappa$.
\item[{\rm (4)}] Each $Y$-Ascoli space $X$ is $Z$-Ascoli for every subspace $Z\subset Y$.
\end{enumerate}
\end{myprop}

\begin{proof}
(1) Let $Z$ be a retract of a $Y$-Ascoli space $X$ and $r:X\to Z$ be a retraction of $X$ onto $Z$. This retraction induces a continuous extension operation $r^*:\CC(Z,Y)\to\CC(X,Y)$, $r^*:f\mapsto f\circ r$.
To show that the space $Z$ is $Y$-Ascoli, fix a compact subspace $\KK\subset\CC(Z,Y)$, a function $f\in\KK$, a point $z\in Z$ and a neighborhood $O_{f(z)}\subset Y$ of $f(z)$. Consider the function $f\circ r:X\to Y$ and the compact subset $r^*(\KK)\subset \CC(X,Y)$. Since the space $X$ is $Y$-Ascoli, there are a neighborhood $O_z\subset X$ of $z$ and a neighborhood $U_{f\circ r}\subset r^*(\KK)\subset\CC(X,Y)$ of $f\circ r$ such that $U_{f\circ r}(O_z)\subset O_{f(z)}$. By the continuity of the extension operation $r^*$, there is a neighborhood $U_f\subset \CC(Z,Y)$ such that $r^*(U_f)\subset U_{f\circ r}$. Then for the neighborhood $Z\cap O_z$ of $z$ in $Z$ we get $U_f(Z\cap O_z)=U_{f\circ r}(Z\cap O_z)\subset U_{f\circ r}(O_z)\subset O_{f(z)}$, which means that the compact set $\KK\subset\CC(Z,Y)$ is evenly continuous. Thus the space $Z$ is $Y$-Ascoli.
\smallskip

(2) Let $\{X_i\}_{i\in I}$ be a family of $Y$-Ascoli spaces. To show that the topological sum $X=\bigoplus_{i\in I}X_i$ is $Y$-Ascoli, fix a compact set $\KK\subset \CC(X,Y)$, a function $f\in\KK$, a point $x\in X$ and a neighborhood $O_{f(x)}$ of $f(x)$ in $Y$. Find an index $i\in I$ such that $x\in X_i$. The continuity of the restriction operator $R_i:\CC(X,Y)\to \CC(X_i,Y)$, $R_i:f\mapsto f|_{X_i}$, implies that the subset $\KK_i=R_i(\KK)\subset \CC(X_i,Y)$ is compact. Since the space $X_i$ is $Y$-Ascoli, there exist a  neighborhood $U_{f|X_i}\subset \CC(X_i,Y)$ of $f|_{X_i}$ and a neighborhood $O_x\subset X_i$ of $x$ such that $U_{f|X_i}(O_x)\subset O_{f(x)}$. By the continuity of the restriction operator $R_i:\CC(X,Y)\to \CC(X_i,Y)$, the set $U_f=\{g\in \CC(X,Y):g|_{X_i}\in U_{f|X_i}\}$ is a neighborhood of $f$. Then $U_f(O_x)=U_{f|X_i}(O_x)\subset O_{f(x)}$, which means that $\KK$ is evenly continuous and $X$ is $Y$-Ascoli.
\smallskip

(3) For $j\in\kappa$, put $p_j:\CC(X,Y^\kappa)\to\CC(X,Y)$, $p_j:f\mapsto \pi_j\circ f$, where $\pi_j :Y^\kappa \to Y$ stands for the projection onto the $j$th coordinate. Then $p_j$ is continuous. Let $\psi_\kappa :\CC(X,Y^\kappa)\times X \to Y^\kappa$ and $\psi:\CC(X,Y)\times X \to Y$ be the evaluation maps. Then $\pi_j \circ\psi_\kappa |_{\KK\times X} =\psi \circ (p_j, \mathrm{id}_X)  |_{\KK\times X}$ is continuous for every $j\in\kappa$ and every compact subset $\KK\subset \CC(X,Y^\kappa)$, and therefore $\psi_\kappa$ is continuous on $\KK\times X$. Thus $X$ is $Y^\kappa$-Ascoli.

\smallskip

(4) The last statement follows from the fact that for every subspace $Z\subset Y$ the function space $\CC(X,Z)$ is a subspace of $\CC(X,Y)$.
\end{proof}

Since any (zero-dimensional) Tychonoff space embeds into some power $\mathbb{R}^\kappa$ (respectively, $2^\kappa$), Proposition \ref{p:AscoliCat} implies
\begin{mycorollary}\label{l:AscoliCat}
\begin{enumerate}
\item[{\rm (1)}] If $X$ is an Ascoli space, then $X$ is $Y$-Ascoli for every Tychonoff space $Y$.
\item[{\rm (2)}] If $X$ is a $2$-Ascoli space, then $X$ is $Y$-Ascoli for every zero-dimensional $T_1$-space $Y$.
\end{enumerate}
\end{mycorollary}

It turns out that the $Y$-Ascoli property of $X$ is responsible for the continuity of the canonical map $\delta:X\to\CC(\CC(X,Y),Y)$.

\begin{myprop} \label{p:AscoliChar}
For topological spaces $X$ and $Y$, the canonical map $\delta: X\to \CC(\CC(X,Y), Y)$ is continuous if and only if $X$ is $Y$-Ascoli.
\end{myprop}
\begin{proof}
Assume that $\delta$ is continuous. To show that $X$ is $Y$-Ascoli, we have to check that every compact subset $\KK$ of $\CC(X,Y)$ is evenly continuous. Fix $f\in \KK$, $x\in X$ and an open neighborhood $O_{f(x)}\subset Y$ of $f(x)$.
Using the regularity of $Y$, choose an open neighborhood $\widetilde{O}_{f(x)}$ of $f(x)$ such that $\mathrm{cl}_Y \big( \widetilde{O}_{f(x)}\big) \subset O_{f(x)}$. Let $U_f := \big[ \{ x\}; \widetilde{O}_{f(x)}\big] \cap\KK$ and $\mathcal{C}:= \mathrm{cl}_\KK (U_f \cap \KK)$. Then for every $g\in\mathcal{C}$ we have
\[
\delta_x(g)=g(x) \in \mathrm{cl}_Y \big( \widetilde{O}_{f(x)}\big) \subset O_{f(x)},
\]
and therefore $\delta_x \in [\mathcal{C}; O_{f(x)}]$. Since $\delta$ is continuous, there is a neighborhood $O_x$ of $x$ such that $\delta(O_x)\subset \big[\mathcal{C}; O_{f(x)}\big]$. So for every $y\in O_x$ and each $g\in U_f \subset \mathcal{C}$ we have $g(y)=\delta_y(g)\in O_{f(x)}$, which means that $\KK$ is evenly continuous.
\smallskip

Conversely, assume that $X$ is $Y$-Ascoli. We have to show that $\delta$ is continuous at each point $x_0\in X$. Fix a sub-basic neighborhood
\[
[\KK;V]\subset\CC(\CC(X,Y),Y)
\]
of $\delta_{x_0}$ with $\KK\subset \CC(X,Y)$ compact and $V\subset Y$ open. It follows from $\delta_{x_0}\in[\mathcal K;V]$ that $f(x_0)=\delta_{x_0}(f)\in V$ for every $f\in\KK$. Since $X$ is $Y$-Ascoli, for every $f\in\KK$  there exist neighborhoods $U_f\subset \KK$ of $f$ and $O_f\subset X$ of $x_0$ such that $U_f(O_f)\subset V$. By the compactness of $\KK$, there is a finite subset $\F\subset\KK$  such that $\KK =\bigcup_{f\in\F} U_f$. Consider the neighborhood $O_{x_0} := \bigcap_{f\in\F} O_{f}$. Then for every $x\in O_{x_0}$ and each $g\in\KK$ we have
\[
\delta_x(g)\in g(O_{x_0})\subset \bigcup_{f\in\F}U_{f}(O_{x_0}) \subset \bigcup_{f\in\F}U_f(O_f)\subset V.
\] This means that $\delta_x\in[\KK;V]$ and hence the canonical map $\delta$ is continuous at $x_0$.
\end{proof}

Next, we give some conditions on topological spaces $X$ and $Y$ guaranteeing that the canonical map $\delta:X\to \delta(X)\subset \CC(\CC(X,Y),Y)$ is injective or open. We shall say that a map $f:X\to Y$ between topological spaces is {\em open} if for any open set $U\subset X$ the image $f(U)$ is open in the subspace $f(X)$ of $Y$.

\begin{mydefinition}\label{d:sep}
Let $X$ and $Y$ be topological spaces. The topological space $X$ is called
\begin{itemize}
\item[$\bullet$] {\em $Y$-separated} if for any distinct points $x,y\in X$ there is a continuous map $f:X\to Y$ such that $f(x)\ne f(y)$;
\item[$\bullet$] {\em $Y$-regular} if for any point $x\in X$ and a neighborhood $O_x\subset X$ of $x$ there is a continuous map $f:X\to Y$ such that $x\in f^{-1}(V)\subset O_x$ for some open set $V\subset Y$;
\item[$\bullet$] {\em $Y$-Tychonoff} if $X$ is $Y$-separated and $Y$-regular.
\end{itemize}
\end{mydefinition}

Observe that each $Y$-Tychonoff space $X$ embeds into some power $Y^\kappa$ of the space $Y$.
Observe also that a topological space $X$ is Tychonoff if and only if $X$ is $\mathbb R$-Tychonoff, and $X$ is zero-dimensional if and only if it is 2-regular  for the doubleton $2=\{0,1\}$ endowed with the discrete topology.

The following proposition can be easily derived from Definition~\ref{d:sep} and Proposition~\ref{p:AscoliChar}.

\begin{myprop}\label{p:AscoliInj} For topological spaces $X,Y$ the canonical map $\delta:X\to \delta(X)\subset \CC(\CC(X,Y),Y)$ is
\begin{enumerate}
\item[{\rm (1)}] injective if and only if $X$ is $Y$-separated;
\item[{\rm (2)}] open if $X$ is $Y$-regular;
\item[{\rm (3)}] open and injective if $X$ is $Y$-Tychonoff;
\item[{\rm (4)}] continuous if and only if $X$ is $Y$-Ascoli;
\item[{\rm (5)}] a topological embedding if $X$ is $Y$-Ascoli and $Y$-Tychonoff.
\end{enumerate}
\end{myprop}

\begin{myremark}
Each connected space $X$ is 2-Ascoli since $\CC(X,2)=\{ \mathbf{0},\mathbf{1}\}$. It follows that the canonical map $\delta:X\to \CC(\CC(X,2),2)$ is constant. Hence $\delta$ is continuous and open, but it is not injective if $|X|>1$.
\end{myremark}

Since the class of $k$-spaces is (properly) contained in the class of Ascoli spaces, the following corollary generalizes Theorem 2.3.6 of \cite{mcoy}.
\begin{mycorollary}
For every Ascoli Tychonoff space $X$ the canonical map $\delta: X\to\CC(\CC(X))$ is a topological embedding.
\end{mycorollary}

Our interest to the study of Ascoli spaces can be explained by the following theorem which gives a ``lower bound'' on the $\CC$-stable closure $\cc[\X]$ of any non-empty class $\X$ of topological spaces. The second part of this theorem is Theorem \ref{t-Ascoli}.

\begin{mytheorem}\label{t:Ascoli-CX}
For any non-empty class $\X$ of topological spaces the class $\cc[\X]$ contains all zero-dimensional 2-Ascoli $\aleph_0$-spaces. If $\II\in\X$, then the class $\cc[\X]$ contains all Ascoli $\aleph_0$-spaces.
\end{mytheorem}

\begin{proof}
By Proposition \ref{p:Ascoli}, the class $\X$ contains all zero-dimensional separable metrizable spaces. Now take any zero-dimensional 2-Ascoli $\aleph_0$-space $X$ and consider the canonical map $\delta:X\to \CC(\CC(X,2),2)$. By Proposition \ref{p:AscoliInj}, $\delta$ is an embedding.  Applying Proposition \ref{p:Ascoli}(3) twice, we obtain that $\CC(\CC(X,2),2)\in \cc[\X]$. Thus also $X\in \cc[\X]$.

Now assume that $\II\in \X$. Given any Ascoli $\aleph_0$-space $X$, we apply Proposition \ref{p:AscoliInj} to conclude that the canonical map $\delta:X\to \CC(\CC(X,\II),\II)$ is a topological embedding. By Proposition \ref{p:Ascoli},  the double function space $\CC(\CC(X,\II),\II)$ and hence also its subspace $X$ belong to the class $\cc[\X]$.
\end{proof}



Below we propose a characterization of Ascoli spaces which will be applied for constructing an Ascoli space $X$ which is not a $k_\IR$-space in Example~\ref{exa:Ascoli-Y-k-cont}.

\begin{myprop}\label{p:Ascoli-char}
A topological space $X$ is Ascoli if and only if each point $x\in X$ is contained in a dense Ascoli subspace of $X$.
\end{myprop}

\begin{proof}
The ``only if'' part is trivial. To prove the ``if'' part, assume that each point $x\in X$ is contained in a dense Ascoli subspace of $X$. Choose any compact set $\KK\subset C_k(X)$, function $f\in\KK$, point $x\in X$, and  neighborhood $O_{f(x)}\subset \IR$ of $f(x)$. By the regularity of $\IR$, there is a neighborhood $W_{f(x)}\subset \IR$ of $f(x)$ such that $\cl_{\IR}(W_{f(x)})\subset O_{f(x)}$. By our assumption, the point $x$ is contained in a dense Ascoli subspace $Z\subset X$. The density of $Z$ in $X$ implies that the restriction operator
\[
\zeta:C_k(X)\to C_k(Z), \quad \zeta:g\mapsto g|_Z,
\]
is injective. Since the space $Z$ is Ascoli, for the compact subset $\zeta(\mathcal K)\subset C_k(Z)$, the function $h:=f|_Z$ and the neighborhood $W_{f(x)}$ of $h(x)=f(x)$ there are neighborhoods $U_h\subset \zeta(\mathcal K)$ of $h$ and $W_x\subset Z$ of $x$  such that $U_h(W_x)\subset W_{f(x)}$. It follows that $U_f :=\{g\in \mathcal{K}:g|_Z\in U_h\}$ is a neighborhood of $f$ in $\KK$ and the closure $\overline{W}_x$ of $W_x$ in $X$ is a (closed) neighborhood of $x$ in $X$ such that $U_f(\overline{W}_x)\subset \overline{W}_{f(x)}\subset O_{f(x)}$. Thus $\KK$ is evenly continuous.
\end{proof}

\vspace{1mm}

Now we prove that the classes of Ascoli and $k_\IR$-spaces are hereditary with respect to taking closed subspaces in stratifiable spaces. We recall (see \mbox{\cite[\S5]{gruenhage}}) that a regular topological space $X$ is {\em stratifiable} if there is a function $G$ which assigns to every $n\in\w$ and each closed set $F\subset X$ an open neighborhood $G(n,F)\subset X$ of $F$ such that $F=\bigcap_{n\in\w} \overline{G(n,F)}$ and $G(n,F)\subset G(n,F')$ for any $n\in\w$ and closed sets $F\subset F'\subset X$. Reznichenko in \cite{Rez} proved that for a separable metrizable space $X$ the function space $\CC(X)$ is stratifiable if and only if the space $X$ is Polish.

\begin{myprop}\label{p:ascoli}
Let $A$ be a closed subspace of a stratifiable space $X$. Then
\begin{enumerate}
\item[{\rm (1)}] if $X$ is a $k_\IR$-space, then $A$ is a $k_\IR$-space;
\item[{\rm (2)}] if $X$ is Ascoli, then $A$ is Ascoli;
\item[{\rm (3)}] if $X$ is $2$-Ascoli and $|X\setminus A|<\mathfrak{c}$, then $A$ is 2-Ascoli.
\end{enumerate}
\end{myprop}

\begin{proof}
To prove the  proposition we need the following construction due to Borges (see Section 4 and the proof of Theorem 4.3 \cite{Bor}).
Let $[A]^{<\w}$ be the set of non-empty finite subsets of $A$ and $\mathsf P_\w(A)$ be the space of probability measures with finite support on $A$. The space $\mathsf P_\w(A)$ is considered with the weak-star topology inherited from the dual space $C^*_b(A)$ of the Banach space $C_b(A)$ of all bounded continuous real-valued functions on $A$ endowed with the sup-norm. The weak-star topology on the dual Banach space $C^*_b(A)$ is inherited from the Tychonoff product topology on $\IR^{C_b(A)}$.
A map $u:X\to [A]^{<\w}$ is called {\em upper semicontinuous} if for any open set $U\subset A$ the set $\{x\in X:u(x)\subset U\}$ is open in $X$.
It was shown in (the proof of Theorem 4.3 of)  \cite{Bor}  that there exist an upper semicontinuous map $u:X\to [A]^{<\w}$ with $u(a)=\{a\}$ for all $a\in A$, and a continuous map $\mu:X\to \mathsf P_\w(A)$ such that for each point $x\in X$  the finite set $u(x)$ has measure 1 with respect to the measure $\mu(x)$ which will be denoted by $\mu_x$. Since for any point $a\in A$ the set $u(a)$ coincides with the singleton $\{a\}$, the measure $\mu_a$ coincides with the Dirac measure $\delta_a$ concentrated at $a$. Then the map $\mu:X\to\mathsf P_\w(X)$ induces a continuous linear  operator $e:\CC(A)\to \CC(X)$ assigning to each function $f\in \CC(A)$ the function $e(f)\in \CC(X)$, $e(f):x\mapsto \mu_x(f)$ (the value $e(f)(x)=\mu_x(f)$ is well-defined as the measure $\mu_x$ has finite support).  Observe that for every point $a\in A$ we get $e(f)(a)=\mu_a(f)=\delta_a(f)=f(a)$, so $e$ is an extension operator.
\smallskip

(1)  Assume that the space $X$ is a $k_\IR$-space. To show that $A$ is a $k_\IR$-space, take any function $f:A\to \IR$ such that for every compact subset $K\subset A$ the restriction $f|_A$ is continuous. Consider the function $\bar f:X\to\IR$ defined by $\bar f(x):= e(f)(x)=\mu_x(f)$ for $x\in X$. We show that for every compact subset $K\subset X$ the restriction $\bar f|_K$ is continuous. The upper semicontinuity of the map $u$ and the compactness of $K$ imply the compactness of the set
\[
u[K]:= \bigcup_{x\in K}u(x)\subset A.
\]
By our assumption, the restriction $f|_{u[K]}$ is continuous. Then the continuity of the map $\mu|_K:K\to \mathsf P_\w(u[K])$ guarantees  the continuity of the function $\bar f|_K$ (as $\bar f(x)=\mu_x(f|_{u[K]})$ for any $x\in K$).
Taking into account that $X$ is a $k_\IR$-space, we conclude that the function $\bar f:X\to\IR$ is continuous and so is its restriction $f=\bar f|_A$.
\smallskip

(2) Now assume that $X$ is an Ascoli space. To show that the space $A$ is Ascoli, fix any compact subset $\K\subset \CC(A)$. Given a function $f\in \K$, a point $a\in A$ and a neighborhood $O_{f(a)}\subset\IR$ of $f(a)$, we need to find neighborhoods $U_f\subset\K$ of $f$ and $O_a\subset A$ of $a$ such that $U_f(O_a)\subset O_{f(a)}$. As the extension operator $e:\CC(A)\to \CC(X)$ is continuous, the set $\tilde\K:=e(\K)\subset \CC(X)$ is compact.   Since the space $X$ is Ascoli, for the function $\bar f=e(f)\in \tilde\K$ there are neighborhoods $U_{\bar f}\subset \tilde\K$ of $\bar f$ and $\tilde O_a\subset X$ of $a$ such that $U_{\bar f}(\tilde O_a)\subset O_{f(a)}$. By the continuity of the operator $e$, there is a neighborhood $U_f\subset\K$ of $f$ such that $e(U_f)\subset U_{\bar f}$. Then for the neighborhood $O_a=\tilde O_a\cap A$ of $a$ in $A$, we get $U_f(O_a)\subset U_{\bar f}(\tilde O_a\cap A)\subset O_{f(a)}$, which means that the compact set $\K$ is evenly continuous.
\smallskip

(3) Assume that $X$ is 2-Ascoli and $|X\setminus A|<\mathfrak c$. Consider the map $\mu:X\to\mathsf P_\w(A)$, and observe that for every $x\in X$ the set
\[
Y_x=\{\mu_x(B):B\subset X \}
\]
of all possible values of the measure $\mu_x:=\mu(x)$ is finite. Since $|X\setminus A|<\mathfrak c$, the set $Y=\bigcup_{x\in X\setminus A}Y_x\subset[0,1]$ has cardinality $<\mathfrak c$.  By (the proof of) Corollary 6.2.8 in \cite{Eng},
the regular space $Y$ is zero-dimensional. Now Corollary \ref{l:AscoliCat} implies that the 2-Ascoli space $X$ is $Y$-Ascoli.

To show that the space $A$ is 2-Ascoli, fix any compact subset $\K\subset \CC(A,2)$. Given a function $f\in \K$, a point $a\in A$ and a neighborhood $O_{f(a)}\subset 2=\{0,1\}$ of $f(a)$, we need to find neighborhoods $U_f\subset\K$ of $f$ and $O_a\subset A$ of $a$ such that $U_f(O_a)\subset O_{f(a)}$.
Take any neighborhood $\tilde O_{f(a)}\subset \IR$ of $f(a)$ such that $\tilde O_{f(a)}\cap\{0,1\}=O_{f(a)}$. Consider the compact set $\tilde\K:=e(\K)\subset \CC(X)$ and the function $\bar f=e(f)\in \tilde\K\subset \CC(X)$. The definition of the extension operation $e$ implies that $\tilde \K\subset \CC(X,Y)\subset\CC(X,\IR)$. Since the space $X$ is $Y$-Ascoli, there are neighborhoods $U_{\bar f}\subset \tilde\K$ of $\bar f$ and $\tilde O_a\subset X$ of $a$ such that $U_{\bar f}(\tilde O_a)\subset \tilde O_{f(a)}$. By the continuity of the operator $e$, there is a neighborhood $U_f\subset\K$ of $f$ such that $e(U_f)\subset U_{\bar f}$. Then for the neighborhood $O_a=\tilde O_a\cap A$ of $a$ in $A$, we get $U_f(O_a)\subset \{0,1\}\cap U_{\bar f}(\tilde O_a\cap A)\subset \{0,1\}\cap \tilde O_{f(a)}=O_{f(a)}$, which means that the compact set $\K$ is evenly continuous, and $A$ is 2-Ascoli.
\end{proof}

The following proposition will be used also in the last section. Recall that a topological space $X$ is called {\em scattered} if each non-empty subspace of $X$ contains an isolated point.

\begin{myprop}\label{p:2-Ascoli-Disc}
A zero-dimensional 2-Ascoli space $X$ is sequential if it satisfies one of the following conditions:
\begin{enumerate}
\item[{\rm (1)}] each compact subset of $X$ is finite (in this case $X$ is discrete);
\item[{\rm (2)}] $X$ is stratifiable, scattered, and has cardinality $|X|<\mathfrak c$.
\end{enumerate}
\end{myprop}

\begin{proof}
(1) Assume that the 2-Ascoli space $X$ does not contain infinite compact subsets. We shall prove that $X$ is discrete (and hence sequential). Assuming that $X$ is not discrete, fix a non-isolated point $x\in X$. By Zorn's Lemma, there exists a maximal disjoint family $\U$ of non-empty closed-and-open subsets of $X$ which do not contain the point $x$. By the maximality of $\U$, the point $x$ belongs to the closure of the union $\bigcup\U$ in $X$. Since all compact subsets of $X$ are finite, the set of characteristic functions $K=\{\chi_\emptyset\}\cup\{\chi_U:U\in\U\}\subset \CC(X,2)$ is compact with a unique non-isolated point $\chi_\emptyset$. Since $X$ is 2-Ascoli, the set $\K$ is evenly continuous. Consequently, we can find a neighborhood $U_{\chi_\emptyset}\subset\K$ of the constant zero function $\chi_\emptyset:X\to 2$ and a  neighborhood $O_x\subset X$ of the point $x$ such that $U_{\chi_\emptyset}(O_x)\subset\{0\}$.

As $\chi_\emptyset$ is a unique non-isolated point of the compact set $\K$, the set $\K\setminus U_{\chi_{\emptyset}}$ is finite. Since the neighborhood $O_x$ meets infinitely many sets $U\in\U$, we can find a set $U\in\U$ such that $O_x\cap U\ne\emptyset$ and $\chi_U\in U_{\chi_\emptyset}$. Then for any point $u\in O_x\cap U$ we get $1=\chi_U(u)\in U_{\chi_\emptyset}(O_x)\subset\{0\}$, a  contradiction. Thus the space $X$ is discrete (and sequential).
\smallskip

(2) Assume that $X$ is a 2-Ascoli stratifiable scattered space of cardinality $|X|<\mathfrak c$. We need to prove that the space $X$ is sequential.
Since all compact subsets of the stratifiable space $X$ are metrizable (see \cite[4.7 and 5.9]{gruenhage}), it suffices to show that $X$ is a $k$-space. Assuming the opposite, we can find a non-closed subset $A\subset X$ such that for each compact subset $K\subset X$ the intersection $K\cap A$ is compact. By Proposition~\ref{p:ascoli}, the closed subspace $\bar A$ of the stratifiable 2-Ascoli space $X$ is 2-Ascoli.
Replacing the space $X$ by the closure $\bar A$ of $A$ in $X$, we can assume that the set $A$ is dense in $X$.

For a subset $B\subset X$ let $B^{(1)}$ denote the set of all non-isolated points of $B$. Let $X^{(0)}=X$, and for every ordinal $\alpha>0$ put $X^{(\alpha)}=\bigcap_{\beta<\alpha}(X^{(\beta)})^{(1)}$. Since $X$ is scattered, for some ordinal $\alpha$ the set $X^{(\alpha)}$ is empty. So for each point $x\in X$ we can assign the unique ordinal $\hbar(x)$ (called the {\em scattered height of $x$}) such that $x\in X^{(\hbar(x))}\setminus X^{(\hbar(x)+1)}$. Consider the ordinal $\alpha=\min\{\hbar(x):x\in X\setminus A\}$ and choose a point $a\in X\setminus A$ with $\hbar(a)=\alpha$. Since $a$ is an isolated point of the set $X^{(\alpha)}\setminus X^{(\alpha+1)}$ and $X$ is zero-dimensional,
we can find a closed-and-open neighborhood $O_a\subset X$ of $a$ such that $O_a\cap X^{(\alpha)}=\{a\}$. This means that each point $x\in O_a\setminus\{a\}$ has scattered height $\hbar(x)<\alpha$. The definition of the ordinal $\alpha$ implies that $O_a\setminus A=\{a\}$ and hence $O_a\cap A=O_a\setminus\{a\}$.

The space $Z:=O_a\setminus\{a\}$ is stratifiable, and hence paracompact \cite[5.7]{gruenhage}. By (the proof of) Corollary 6.2.8 of \cite{Eng}, the space $Z$ of cardinality $|Z|<\mathfrak c$ is strongly zero-dimensional and hence has covering dimension $\dim(Z)=0$.
Since $Z$ is paracompact and has covering dimension zero we can apply Dowker's Theorem \cite[7.2.4]{Eng} and conclude that the open cover
\[
\mathcal{W}=\{ Z\setminus U: \ U \mbox{ is a clopen neighborhood of } a \mbox{ in } O_a\}
\]
of $Z$ has a disjoint open refinement $\mathcal{V}$ covering $Z$. Observe that each element $V\in \mathcal{V}$ is  closed in $O_a$ since $V=(O_a\setminus U)\setminus \bigcup\{ V'\in \mathcal{V}: V' \not= V\}$ for some clopen subset $U$ in $O_a$.
Since $a$ is an accumulation point of $O_a\setminus\{a\}$, each neighborhood $U_a\subset X$ of $a$ meets infinitely many (clopen) sets $V\in\V$.

We claim that the set  of characteristic functions
\[
\K=\{\chi_\emptyset\}\cup\{\chi_V:V\in\V\}
\]
is compact in $\CC(X,2)$ and has a unique non-isolated point $\chi_\emptyset$. It suffices to check that any neighborhood $O_{\chi_\emptyset}\subset \CC(X,2)$ of $\chi_\emptyset$ contains all but finitely many functions $f\in\K$. Without loss of generality we can assume that the neighborhood $O_{\chi_\emptyset}$ is of the basic form
\[
O_{\chi_\emptyset}=\big\{f\in \CC(X,2):f(K)\subset\{0\}\big\},
\]
for some compact set $K\subset X$ containing the point $a$. By the choice of the set $A$, the intersection $A\cap K$ is compact and so is its closed subset $(O_a\cap A)\cap K= (O_a\setminus\{a\})\cap K$.
This means that $a$ is an isolated point of the compact space $K$. Since $\V$ is a disjoint open cover of $O_a\setminus \{a\}$, the compactness of $(O_a\setminus\{a\})\cap K$ guarantees that $K$ meets only finitely many sets $V\in\V$. This implies that the neighborhood $O_{\chi_\emptyset}$ contains all but finitely many characteristic functions $\chi_V$, $V\in\V$. So, $\K\subset \CC(X,2)$ is a compact set with the unique non-isolated point $\chi_\emptyset$.

Since the space $X$ is 2-Ascoli, the compact set $\K$ is evenly continuous. This allows us to find a neighborhood $U_{\chi_\emptyset}\subset\K$ of the constant zero function $\chi_\emptyset$ and a neighborhood $W_a\subset X$ of the point $a$ such that $U_{\chi_\emptyset}(W_a)\subset\{0\}$. Since $\chi_\emptyset$ is a unique non-isolated point of the compact set $\K$ the set
\[
\V'=\{V\in\V:\chi_V\in U_{\chi_\emptyset}\}
\]
has finite complement $\V\setminus\V'$. As each neighborhood of $a$ meets infinitely many sets $V\in\V$, we can find a set $V\in\V'$ such that the intersection $W_a\cap V$ contains some point $v$. Then $1=\chi_V(v)\in U_{\chi_\emptyset}(W_a)\subset\{0\}$, a contradiction. Thus the space $X$ is sequential.
\end{proof}

Now we present an example of two Fr\'echet-Urysohn stratifiable $\aleph_0$-spaces $X$ and $Y$ whose product $X\times Y$ is 2-Ascoli but not Ascoli.
The space $X$ is the following $\sigma$-compact subspace of the real plane
\[
X:=\bigcup_{n\in\w\setminus\{ 0\} }\{(t,t/n):t\in[0,1]\}\subset \IR^2,
\]
which is called the {\em connected metric fan}. The space $Y$ is the space $X$ endowed with the strongest topology inducing the Euclidean topology on each arc
\[
I_n=\{(t,t/n):t\in[0,1]\}, \quad n>0.
\]
The space $Y$ is called the {\em connected Fr\'echet-Urysohn fan} and is a (non-metrizable) Fr\'echet-Urysohn $k_\omega$-space.

The following proposition shows that the class of Ascoli spaces is not productive and the class of  2-Ascoli spaces is neither productive nor closed hereditary.

\begin{myprop}\label{p:notascoli}
The spaces $X$ and $Y$ have the following properties:
\begin{enumerate}
\item[{\rm (1)}] $X$ is separable and  metrizable,  while $Y$ is a Fr\'echet-Urysohn stratifiable $\aleph_0$-space;
\item[{\rm (2)}] the spaces $X$ and $Y$ are Ascoli;
\item[{\rm (3)}] the product $X\times Y$ is 2-Ascoli but is not Ascoli;
\item[{\rm (4)}] the spaces $X$ and $Y$ contain closed countable scattered subspaces $X_0\subset X$ and $Y_0\subset Y$ whose product $X_0\times Y_0$ is not 2-Ascoli.
\end{enumerate}
\end{myprop}

\begin{proof}
The metrizability and separability of $X$ is clear. It follows that the space $Y$ is the image of the metrizable separable space $\oplus_{n> 0}I_n$ under a closed compact-covering map.  This implies that $Y$ is a stratifiable $\aleph_0$-space, see Theorem 5.5 in \cite{gruenhage}. Then the product $X\times Y$ is a stratifiable $\aleph_0$-space (by the productivity of the classes of stratifiable and $\aleph_0$-spaces, see \cite[5.10 and 11.2]{gruenhage}). The spaces $X$ and $Y$ are Ascoli spaces (being $k$-spaces). The connectedness of the space $X\times Y$ implies that this space is 2-Ascoli.

In the spaces $X$ and $Y$ consider the countable subspaces
\[
X_0=\{x_{\infty}\}\cup\{x_{n,m}:n,m>0\}\subset X \; \mbox{ and }\;  Y_0=\{y_\infty\}\cup\{y_{n,m}:n,m>0\} \subset Y,
\]
where
\[
x_\infty=(0,0)=y_{\infty}, \; x_{n,m}=(\tfrac1n,\tfrac1{nm}) \quad  \mbox{ and } \quad  y_{n,m}=(\tfrac1m,\tfrac1{mn}), \;  \forall n,m>0.
\]
It follows that for every $n>0$ the set $\{x_{n,m}\}_{m>0}$ is closed and discrete in $X$, while the sequence $\{y_{n,m}\}_{m>0} \subset I_n$  converges  to $y_\infty$ in $Y$. It is known (see, e.g. \cite{Ba98}) that the space $X_0\times Y_0$ is not sequential.  By Proposition~\ref{p:2-Ascoli-Disc}(2), the countable scattered space $X_0\times Y_0$ is not 2-Ascoli and hence not Ascoli. Since $X_0\times Y_0$ is a closed subspace of the stratifiable space $X\times Y$, we can apply Proposition~\ref{p:ascoli} to conclude that the space $X\times Y$ is not Ascoli.
\end{proof}


\begin{myproblem}
Is each zero-dimensional 2-Ascoli space Ascoli?
\end{myproblem}



\section{Examples and Open Questions}


In this section we provide examples which show that the implications in the diagram presented in the introduction cannot be reversed and pose several open questions.

\begin{myexample}
There exists an Ascoli $\aleph_0$-space which is a $k_\IR$-space but is not a $k$-space.
\end{myexample}

\begin{proof}
In \cite{Mi73} Michael constructed an $\aleph_0$-space $X$, which is a $k_\IR$-space but is not a $k$-space (and hence is not sequential). By \cite{Noble}, the space $X$ is Ascoli.
\end{proof}

\begin{myexample}
The class $\cc[\M_0]$ contains a countable topological group $\Delta$ such that
\begin{itemize}
\item[$\bullet$] $\Delta$ is not discrete but all compact subsets of $\Delta$ are finite;
\item[$\bullet$] $\Delta$ embeds into the product $F\times G$ of a countable $k_\omega$-group $F$ and\newline a metrizable group $G$;
\item[$\bullet$] $\Delta$ is stratifiable;
\item[$\bullet$] $\Delta$ is not 2-Ascoli and hence is not Ascoli.
\end{itemize}
\end{myexample}

\begin{proof}
Let $F$ be the free abelian topological group over the convergent sequence $\{ 0\}\cup\{\frac{1}{n} : n>0\}\subset\IR$. It is well-known that $F$ is a countable $k_\w$-space and hence $F$ is a sequential stratifiable $\aleph_0$-space (for the stratifiability of $F$, see \cite[5.5]{gruenhage}). By Theorem~\ref{t-Ascoli}, the topological space $F$ belongs to the class $\cc(\M_0)$ (recall that any sequential space is Ascoli).

Denote by $G$ the free abelian group $F$, endowed with the metrizable group topology $\, \tau \,$ whose $\,$ neighborhood $\,$ base at zero consists of the subgroups \mbox{$2^kF$, $k\in\w$}. Being metrizable and separable, the topological group $G$ is a  \mbox{$\CC[\M_0]$-space}. By \cite{BaT}, the diagonal subgroup $\Delta=\{(x,y)\in F\times G:x=y\}$ is not discrete, and by \cite{Ga30} every compact subset of $\Delta$ is finite. It follows that the space $\Delta$ is not a $k$-space. Applying
Proposition \ref{p:2-Ascoli-Disc}(1),  we conclude that $\Delta$ is not 2-Ascoli (and hence not Ascoli). Since both spaces $F$ and $G$ are stratifiable, so are their product $F\times G$ and the subspace $\Delta\subset F\times G$.
\end{proof}

\begin{myremark}
The stratifiable space $X\times Y$ from Proposition~\ref{p:notascoli} also belongs to the class $\cc[\M_0]$ (by productivity of $\cc[\M_0]$) but fails to be Ascoli.
\end{myremark}

\begin{myexample} \label{exa-Pyt0-C[M]}
There exists a $\Pp_0$-space, which is not a $\CC[\M]$-space.
\end{myexample}

\begin{proof}
Let $X$ be a  separable  metrizable space containing a topological copy of the Cantor cube $2^\w$. Since the Cantor cube is homeomorphic to its own square, the space $X$ contains an uncountable family $\C$ of pairwise disjoint topological copies of $2^\w$.
Enlarge the family $\C$ to the smallest discretely-complete ideal $\I$ of compact subsets of $X$.
Hence any element of $\I$ is contained in the union of a finite subfamily of $\C$ and a discrete subset of $X$. So, the union of any infinite subfamily of $\C$ does not belong to $\I$, and therefore  $\I$ fails to be a $P^{\w_1}_\w$-ideal. Then by Theorem~\ref{t:P}, the function space $C_\I(X)$ fails to be a $P^{\w_1}_\w$-space and hence cannot be a $\CC[\M]$-space. On the other hand, Theorem~\ref{saak} guarantees that the function space $C_\I(X)$ is a $\Pp_0$-space.
\end{proof}

\begin{myremark}
In  \cite{GK-GMS1} it is shown that the precompact group $\mathbb{Z}^\sharp$ of integers  endowed with the Bohr topology is an $\aleph_0$-space but fails to be a $\Pp$-space.
\end{myremark}

\begin{myproblem}
Is there an Ascoli $\aleph$-space which is not a $\CC[\M]$-space?
\end{myproblem}

\begin{myexample} \label{exa:Ascoli-Y-k-cont}
Let $\lambda$ and $\kappa>\lambda$ be infinite cardinals and $Y$ be a Tychonoff first countable space containing more than one point. Then  the subspace
\[
X=\bigcup_{y\in Y}\{f\in Y^\kappa:|f^{-1}(Y\setminus \{y\})|<\lambda\}
\]
of $Y^\kappa$ is Ascoli but fails to be a $k_\IR$-space. If $Y$ is a topological group (or a linear topological space), then so is the space $X$.
\end{myexample}

\begin{proof}
To see that $X$ is Ascoli, it suffices to check that each element $f\in X$ is contained in a dense Ascoli subspace of $X$.
By \cite[3.10.D]{Eng}), the $\sigma$-product $\sigma(f)=\{g\in Y^\kappa:|\{x\in\kappa:f(x)\ne g(x)\}|<\w\}\subset Y^\kappa$ is  Fr\'{e}chet--Urysohn and hence Ascoli according to the Ascoli Theorem \cite[3.4.20]{Eng}. Clearly, $\sigma(f)$ is a dense subset of $X$, and therefore $X$ is Ascoli by Proposition~\ref{p:Ascoli-char}.

To show that the space $X$ is not a $k_\IR$-space, consider the map $\lim:X\to Y$ assigning to each function $f\in X$ the unique point $y\in Y$ such that the set $\supp(f):=f^{-1}(Y\setminus\{y\})$ has cardinality $|\supp(f)|<\lambda$. It is clear that the map $\lim $ is discontinuous. We claim that for every compact subset $K\subset X$ the restriction $\lim|_K$ is continuous, and therefore $X$ is not a $k_\IR$-space.

It suffices to prove that, for each closed subset $B\subset Y$, the set $D:=\{f\in K:\lim f\in B\}$ is closed in $K$. Suppose for a contradiction that the set $D$ is not closed in $K$ and has an accumulation point $f\in K\setminus D$. Then $y=\lim f\notin B$. Take any subset $N\subset \kappa\setminus \supp(f)$ of cardinality $|N|=\lambda$. It follows from $N\cap\supp(f)=\emptyset$ that $f(N)=\{y\}$. Let $\FF(N)$ be the family of all finite subsets of $N$, partially ordered by the inclusion relation.

By the regularity of $Y$, the point $y$ has a closed neighborhood $\bar O_y\subset Y$, disjoint with the closed set $B$. Since $f$ is an accumulation point of the set $D$, for every finite subset $F\in\FF(N)$ we can choose a function $f_F\in D\subset K$ such that $f_F(F)\subset \bar O_y$.  By the compactness of $K$, the net $(f_F)_{F\in\FF(N)}$ has a limit point $f_\infty\in K$.  So for any neighborhood $O(f_\infty)\subset X \subset Y^\kappa$ of $f_\infty$ and any $F\in\FF(N)$ there is an element $E\in\FF(N)$ such that $F\subset E$ and $f_E\in O(f_\infty)$, and therefore  $f_\infty(N)\subset\bar O_y$.

Now consider the set $S:=\bigcup_{F\in\FF(N)}\supp(f_F)$ and observe that $|S|\le|N|<\kappa$.  Since the set $\{f_F\}_{F\in\FF(N)}$ is contained in the closed subset $\{g\in Y^\kappa:g(\kappa\setminus S)\subset B\}$ of $Y^\kappa$ the limit point $f_\infty$ has the property $f_\infty(X\setminus S)\subset B$. Since the sets $N$ and $X\setminus S$ have cardinality $\geq\lambda$ and $f_\infty(N)\cap f_\infty(X\setminus S)\subset \bar O_y\cap B=\emptyset$, the function $f_\infty$ does not belong to $X$ and hence  $f_\infty \not\in K$. This contradiction completes the proof of the continuity of the restriction $\lim|_K$.
\end{proof}

The Ascoli space constructed in Example~\ref{exa:Ascoli-Y-k-cont} is not cosmic and hence not an $\aleph_0$-space.

\begin{myproblem}\label{prob:Ascoli->k}
Is there an Ascoli space $X$ which is cosmic (or an $\aleph_0$-space) but fails to be a $k_\IR$-space?
\end{myproblem}


In Theorem~\ref{t-Main-C(M)} we characterized $\CC[\M_0]$-spaces via embeddings into function spaces between separable metrizable spaces.

\begin{myproblem}
Give an inner characterization of $\CC[\M_0]$-spaces (desirably, in terms of special networks).
\end{myproblem}

By definition, the class $\cc[\M]$ is closed under countable topological sums.

\begin{myproblem}\label{prob21}
Is the class $\cc[\M]$ closed under taking arbitrary topological sums?
\end{myproblem}

This problem is related to another open problem. Let $\X$ be a class of topological spaces.
A topological space $U\in\X$ is called {\em universal} in $\X$ if $U$ contains a topological copy of any space $X\in \X$. Since any discrete space $D$ belongs to $\cc[\M]$, the class $\cc[\M]$ does not have universal spaces.

\begin{myproblem}\label{prob22}
Is there a universal space $U$ in the class $\cc[\M_0]$?
\end{myproblem}

The next proposition describes a relation between Problems \ref{prob21} and \ref{prob22}.

\begin{myprop}
If the class $\cc[\M]$ is closed under taking topological sums, then the class $\cc[\M_0]$ contains a universal space.
\end{myprop}

\begin{proof}
By assumption, for the topological sum $\bigoplus_{X\subset \IR^\w}\CC(X)$ there is a topological embedding $$e:\bigoplus_{X\subset\IR^\w}\CC(X)\hookrightarrow \CC(Z,Y)$$ for some spaces $Z\in \M_0$ and $Y\in\M$. For every subspace $X\subset\IR^\w$, the function space $\CC(X)$ is Lindel\"of and so is its topological copy $e(\CC(X))$ in $\CC(Z,Y)$. Repeating the argument from the proof of Theorem~\ref{t:CM0}, we can find a separable subspace $Y_X\subset Y$ such that $e(\CC(X))\subset \CC(Z,Y_X)$. The space $Y_X$, being separable and metrizable, embeds into the countable product $\IR^\w$ of the real line. Then the function space $\CC(Z,Y_X)$ embeds into $\CC(Z,\IR^\w)$, which is homeomorphic to $\CC(Z\times\w)$. This means that the space $\CC(Z\times\w)$ is universal in the class $\cc[\M_0]$.
\end{proof}

We expect that the answer to Problems~\ref{prob21} and \ref{prob22} are negative.
Let us observe the following two facts.

\begin{myprop}\label{p:last}
Let $X$ and $Y$ be two separable metrizable spaces such that the function space $\CC(X)$ embeds into $\CC(Y)$.
\begin{enumerate}
\item[{\rm (1)}] If\/ $Y$ is locally compact, then so is $X$.
\item[{\rm (2)}] If\/ $Y$ is Polish, then so is the space $X$.
\end{enumerate}
\end{myprop}

\begin{proof}
(1) If $Y$ is a locally compact separable metrizable space, then the function space $\CC(Y)$ is metrizable and so is the space $\CC(X)$. By \cite[4.4.2]{mcoy} and \cite[3.4.E]{Eng}, the space $X$ is locally compact.
\smallskip

(2) If $Y$ is Polish, then by \cite{feka},  the function space $\CC(Y)$ has a $\mathfrak{G}$-base at zero-function $\mathbf{0}\in \CC(Y)$. The latter means that each point of $\CC(Y)$ has a neighborhood base $(U_\alpha)_{\alpha\in \w^\w}$ indexed by functions $\alpha:\w\to\w$ such that $U_\alpha\subset U_\beta$ for any functions $\beta\le\alpha$ in $\w^\w$. The function space $\CC(X)$, being a subspace of $\CC(Y)$ also has a $\mathfrak{G}$-base at each point. In this case Corollary 3 of \cite{feka} implies that the space $X$ is Polish.

The same conclusion could be derived from the Reznichenko's characterization \cite{Rez} of Polish spaces as separable metrizable spaces with stratifiable function spaces $\CC(X)$. Indeed, if $Y$ is Polish, then by \cite{Rez}, the function space $\CC(Y)$ is stratifiable and so is its subspace $\CC(X)$. Applying the Reznichenko's characterization once again, we conclude that the separable metrizable space $X$ is Polish.
\end{proof}

Proposition~\ref{p:last} suggests the following problem (or rather, a program of research).

\begin{myproblem}Let $X,Y$ be two separable metrizable spaces such that the function space $C_k(X)$ embeds into $C_k(Y)$. Which topological properties of $Y$ are inherited by $X$? In particular, if\/ $Y$ belongs to certain Borel or projective class, does then $X$ belong to the same Borel or projective class?
\end{myproblem}

\begin{myremark}
In \cite{Banakh-Ascoli,Gabr-C(X-2),GKP,Pol-2015} Ascoli spaces were detected among function spaces, locally convex spaces, and some spaces appearing in Topological  Algebra.
\end{myremark}

{\bf Acknowledgments} The authors are deeply indebted to Professor R.~Pol for fruitful discussion on the Ascoli property in function spaces. We would like to thank the referee for valuable remarks and suggestions.

\end{document}